\documentclass[11pt,a4paper]{scrartcl}
\usepackage[utf8]{inputenc}
\usepackage{amsmath}
\usepackage{amsfonts}
\usepackage{amssymb}
\usepackage{amsthm}
\usepackage{graphicx}
\usepackage{enumitem}
\usepackage{lmodern}
\usepackage{microtype}
\usepackage[]{scrlayer-scrpage}
\usepackage[english]{babel}
\usepackage[table]{xcolor}
\usepackage{tabularx} 
\usepackage{xspace}		
\usepackage{setspace}   
\usepackage[autostyle]{csquotes} 
\usepackage{xpatch} 
\usepackage{aliascnt} 
\usepackage{changes}

\usepackage[
style=numeric,    		
isbn=false,                
url=false,
doi=false,
pagetracker=true,          
maxbibnames=5,            
maxcitenames=4,            
giveninits=true,			
autocite=inline,           
block=space,               
date=comp,                
dateabbrev=false,			
backend=biber,
]{biblatex}
\DeclareFieldFormat[article,incollection,inproceedings,thesis]{citetitle}{#1\midsentence} 
\DeclareFieldFormat[article,incollection,inproceedings,thesis]{title}{#1\midsentence}
\DeclareFieldFormat{titlecase}{\MakeSentenceCase*{#1}} 
\DeclareFieldFormat{titlecase}{\MakeTitleCase{#1}}

\newrobustcmd{\MakeTitleCase}[1]{%
	\ifthenelse{\ifcurrentfield{booktitle}\OR\ifcurrentfield{booksubtitle}%
		\OR\ifcurrentfield{maintitle}\OR\ifcurrentfield{mainsubtitle}%
		\OR\ifcurrentfield{journaltitle}\OR\ifcurrentfield{journalsubtitle}%
		\OR\ifcurrentfield{issuetitle}\OR\ifcurrentfield{issuesubtitle}%
		\OR\ifentrytype{book}\OR\ifentrytype{mvbook}\OR\ifentrytype{bookinbook}%
		\OR\ifentrytype{booklet}\OR\ifentrytype{suppbook}%
		\OR\ifentrytype{collection}\OR\ifentrytype{mvcollection}%
		\OR\ifentrytype{suppcollection}\OR\ifentrytype{manual}%
		\OR\ifentrytype{periodical}\OR\ifentrytype{suppperiodical}%
		\OR\ifentrytype{proceedings}\OR\ifentrytype{mvproceedings}%
		\OR\ifentrytype{reference}\OR\ifentrytype{mvreference}%
		\OR\ifentrytype{report}\OR\ifentrytype{thesis}}
	{#1}
	{\MakeSentenceCase{#1}}}
\renewrobustcmd*{\bibinitdelim}{\,} 
\xpatchbibdriver{online}{\usebibmacro{date}}{}{}{}
\xpatchbibdriver{online}{ \usebibmacro{addendum+pubstate}}{\addcomma\space\usebibmacro{publisher+location+date}\newunit\newblock\usebibmacro{addendum+pubstate}}{}{} 
\bibliography{bib.bib}  
\usepackage{url} 
\usepackage[hidelinks]{hyperref} 
\usepackage[capitalise,nameinlink,noabbrev]{cleveref}
\crefformat{equation}{#2(#1)#3}

\theoremstyle{definition}
\newtheorem{definition}{Definition}

\newaliascnt{question}{definition}
\newtheorem{question}[question]{Question}
\aliascntresetthe{question}

\theoremstyle{plain}
\newtheorem{theorem}{Theorem}[section]
\newaliascnt{lemma}{theorem}
\newtheorem{lemma}[lemma]{Lemma}
\aliascntresetthe{lemma}
\newaliascnt{corollary}{theorem}
\newtheorem{corollary}[corollary]{Corollary}
\aliascntresetthe{corollary}
\newaliascnt{proposition}{theorem}
\newtheorem{proposition}[proposition]{Proposition}
\aliascntresetthe{proposition}

\theoremstyle{remark}

\addto\extrasenglish{%
}

\newcommand{\F}{\mathbb{F}}

\newcommand{\Fpm}{\mathbb{F}_{2^m}}
\newcommand{\Fptwom}{\mathbb{F}_{2^{2m}}}
\newcommand{\Z}{\mathbb{Z}}
\newcommand{\N}{\mathbb{N}}
\newcommand{\A}{\mathcal{A}}
\newcommand{\Aut}{\textnormal{Aut}}

\newcommand{\Gal}{\textnormal{Gal}}

\def\zhou#1 {\fbox {\footnote {\ }}\ \footnotetext { From Yue: {\color{red}#1}}}

\def\chen#1 {\fbox {\footnote {\ }}\ \footnotetext { From Chen: {\color{blue}#1}}}

\newcommand\blfootnote[1]{%
	\begingroup
	\renewcommand\thefootnote{}\footnote{#1}%
	\addtocounter{footnote}{-1}%
	\endgroup
}

\begin{document}
	\title{The number of almost perfect nonlinear functions grows exponentially}
	\author{Christian Kaspers\thanks{Institute for Algebra and Geometry, Otto von Guericke University Magdeburg, 39106 Magdeburg, Germany (email: \href{mailto:christian.kaspers@ovgu.de}{\nolinkurl{christian.kaspers@ovgu.de}})} \space and Yue Zhou\thanks{Department of Mathematics, National University of Defense Technology,  410073 Changsha, China (email: \href{mailto:yue.zhou.ovgu@gmail.com}{\nolinkurl{yue.zhou.ovgu@gmail.com}})}}
	\date{\today}
	\maketitle
	
	\begin{abstract}
	Almost perfect nonlinear (APN) functions play an important role in the design of block ciphers as they offer the strongest resistance against differential cryptanalysis. Despite more than 25 years of research, only a limited number of APN functions are known. In this paper, we show that a recent construction by Taniguchi provides at least $\frac{\varphi(m)}{2}\left\lceil \frac{2^m+1}{3m} \right\rceil$ inequivalent APN functions on the finite field with ${2^{2m}}$ elements, where $\varphi$ denotes Euler's totient function. This is a great improvement of previous results: for even $m$, the best known lower bound has been $\frac{\varphi(m)}{2}\left(\lfloor \frac{m}{4}\rfloor +1\right)$, for odd $m$, there has been no such lower bound at all. Moreover, we determine the automorphism group of Taniguchi's APN functions.
	\end{abstract}
	
	

\section{Introduction}
\label{sec:introduction}
	\blfootnote{\copyright\ IACR 2020. This article is the final version submitted by the authors to the IACR and to Springer-Verlag on November 26, 2020. The version published by Springer-Verlag is available at: URL follows.}
	A function $f:\F_{2^n} \rightarrow \F_{2^n}$ is called \emph{almost perfect nonlinear} (APN) if the equation
	\[
		f(x+a)+f(x)=b
	\]
	has exactly $0$ or $2$ solutions for any $b\in \F_{2^n}$ and any nonzero $a\in \F_{2^n}$. APN functions were introduced in \citeyear{nyberg1994} by \textcite{nyberg1994}. She defined them as the mappings with the highest resistance to differential cryptanalysis, which is one of the most important cryptanalyst tools for block ciphers and was introduced in \citeyear{biham1991} by \textcite{biham1991}. APN~functions and other functions with low differential uniformity are widely used in the design of symmetric key cryptographic algorithms such as the S-boxes in nonlinear layers of block ciphers. For instance, in the hardware oriented MISTY ciphers \cite{matsui_block_1997}, the 16-bit state is split into two parts of different odd lengths on which two APN permutations are used; in the AES algorithm \cite{AES_2000}, an affine transformation of the inverse function over $\F_{2^8}$ which has differential $4$-uniformity was chosen as the S-box. \textcite{blondeau_perfect_2015} provide an overview of theoretical results and applications of APN functions in cryptography.\par
	
	APN functions are also strongly connected with coding theory and finite geometry. In particular, quadratic APN functions are equivalent to a special type of dimensional dual hyperovals; see the work by \textcite{yoshiara2008,edel2010,dempwolff2014} for more details. 
	
	Since their introduction, APN functions have been studied intensively. For an extended overview of these functions, we refer to the survey by \textcite{pott2016}. For a long time, only very few APN functions were known, all of which power functions of the form $x \mapsto x^d$. In~\citeyear{edel2006}, \textcite{edel2006} reported the first two examples of non-power APN functions on $\F_{2^{10}}$ and $\F_{2^{12}}$. Since then, quite a few infinite families of non-power APN functions have been found. A recent list of them was given by \textcite[Table~3]{budaghyan2020}.\par
	
	Except for some sporadic examples, every known non-power APN function is equivalent to a quadratic APN function, that can be written in the form $\sum_{0\leq i<j\le n-1}\alpha_{i,j} x^{2^i+2^j}+\sum_{0\leq i \le n-1}\beta_i x^{2^i} +\gamma$ with $\alpha_{i,j},\beta_i,\gamma\in \F_{2^n}$ for $i,j = 0, \dots, n-1$ and not all $\alpha_{i,j}=0$. By equivalent we mean there exists a CCZ-equivalence transformation between functions over $\F_{2^n}$. This equivalence relation was introduced in \citeyear{carletcharpinzinoviev1998} by \textcite{carletcharpinzinoviev1998}, it preserves the APN property.\par
	
	When $n$ is odd, several known APN functions are also permutations on $\F_{2^n}$. The most fascinating problem regarding APN functions is to find APN permutations on $\F_{2^{n}}$ where $n$ is even. So far, only one such function is known: it was found by \textcite{dillon2010} on $\F_{2^6}$. This sporadic example is also equivalent to a quadratic APN function.\par
	
	The most general equivalence relations of functions from $\F_2^n$ to $\F_2^n$ is the so-called CCZ-equivalence. A very basic and natural question concerning APN functions is the following.
	
	\begin{question}\label{question}
		How many CCZ-inequivalent APN functions on $\F_{2^n}$ exist for a given $n$?
	\end{question}
	\setcounter{question}{0}
	
	Despite its simplicity, this question has not been satisfactorily answered yet. By checking the known APN functions, see \autoref{sec:known_classes}, we first notice that all the power APN functions only provide very few inequivalent examples. Little is known, however, about the number of inequivalent non-power APN functions as it is, in general, a very hard problem to prove the non-equivalence of two functions. Only for small dimensions, this problem can be solved computationally, for larger dimensions, one has to solve it theoretically. Studying a special family of non-power APN functions introduced by Pott and the second author~\cite{zhou2013}, the present authors~\cite{kasperszhou2020} recently presented a first benchmark to answer \autoref{question} for certain fields: they showed that there are at least $\frac{1}{2}\varphi(m)\left(\lfloor m/4\rfloor +1\right)$ inequivalent APN functions on $\Fptwom$ with $m$ even, where $\varphi$ is Euler's totient function.
	
	In this paper, we considerably improve this lower bound and extend it to $\Fptwom$ for any $m \geq 2$. We investigate a family of APN functions defined on $\F_{2^{2m}}$ for any $m\geq 2$ that has been found by \textcite{taniguchi2019}. By completely determining the equivalence of members among this family, we show that the number of inequivalent APN functions on $\F_{2^{2m}}$ is at least 
	\[
		\frac{\varphi(m)}{2}\left\lceil \frac{2^m+1}{3m} \right\rceil.
	\]
	As a corollary, our results enables us to determine the automorphism group of the Taniguchi APN functions.	

	The paper is organized as follows. In \autoref{sec:Preliminaries}, we introduce all necessary definitions and notations. In \autoref{sec:known_classes}, we give an overview of the known classes of APN functions and introduce the constructions by \textcite{taniguchi2019} and Pott and the second author~\cite{zhou2013}. Afterwards, we solve the equivalence problem for the Taniguchi APN functions and present their automorphism group in \autoref{sec:Taniguchi_equivalence}. In \autoref{sec:Taniguchi_number}, we use these results to establish the aforementioned lower bound on the total number of inequivalent APN functions on $\F_{2^{2m}}$. To conclude, we point out several open problems regarding APN functions in \autoref{sec:conclusion}.

\section{Preliminaries}
\label{sec:Preliminaries}
	In this section, we present all the definitions and basic results needed to follow the paper. Denote by $\F_{2}^n$ the $n$-dimensional vector space over the finite field $\F_2$ with two elements. A function from $\F_2^n$ to $\F_2^m$ is called a \emph{vectorial Boolean function} if $m \ge 2$ or simply a \emph{Boolean function} if $m=1$. In this paper, we will only consider vectorial Boolean functions from $\F_2^n$ to $\F_2^n$, we say functions \emph{on} $\F_2^n$. In most cases, we identify the $n$-dimensional vector space $\F_2^n$ over $\F_{2}$ with the finite field $\F_{2^n}$ with $2^n$ elements. This will allow us to use finite field operations and notations. Note that any function on the finite field $\F_{2^n}$ can be written as a univariate polynomial mapping of degree at most $2^n-1$. Furthermore, denote by $\F_{2^n}^*$ the multiplicative group of $\F_{2^n}$.\par 

	Besides our definition given above, there are several equivalent definitions of almost perfect nonlinear functions. We refer to \textcite{budaghyan2014} and \textcite{pott2016} for an extended overview of these functions. In this paper, we will only consider \emph{quadratic} APN functions. We define this term using the coordinate function representation of a function on $\F_2^n$.

\begin{definition}
	Let $f \colon \F_2^n \to \F_2^n$ be a vectorial Boolean function defined by $n$ Boolean coordinate functions $f_1, \dots, f_n\colon \F_2^n \to \F_2$ that are given in their algebraic normal form, that is 
	\[
		f(x_1,\dots,x_n) = 
		\begin{pmatrix}
		f_1(x_1, \dots, x_n)\\\vdots\\ f_n(x_1, \dots, x_n)
		\end{pmatrix}.
	\]
	The maximal degree of the coordinate functions $f_1, \dots, f_n$ is called the \emph{algebraic degree} of~$f$. We call a function of algebraic degree $2$ \emph{quadratic}, and a function of algebraic degree $1$ \emph{affine}. If $f$ is affine and has no constant term, we call $f$ \emph{linear}.
\end{definition}

In polynomial mapping representation, any quadratic function $f$ on $\F_{2^n}$ can be written in the form
	\[
		f(x) = \sum_{\substack{i,j = 0 \\ i < j}}^{n-1}\alpha_{i,j} x^{2^i+2^j} + \sum_{i = 0}^{n-1}\beta_i x^{2^i} + \gamma,
	\]
	and any affine function $f \colon \F_{2^n} \to \F_{2^n}$ can be written as
	\[
		f(x) = \sum_{i=0}^{n-1}\beta_i x^{2^i} + \gamma.
	\]
	If $f$ is affine and $\gamma = 0$, then $f$ is linear. Similar terms are used to describe polynomials over $\F_{2^n}$. Denote by $\F_{2^n}[X]$ the univariate polynomial ring over $\F_{2^n}$. A polynomial of the form
	\[
		P(X) = \sum_{i\ge 0}\alpha_i X^{2^i}
	\]
	is called a \emph{linearized polynomial}. Note that there is a one-to-one correspondence between linear functions on $\F_2^n$ and linearized polynomials in $\F_{2^n}[X] / (X^{2^n}-X)$.  In the same way as for univariate polynomials, we define a linearized polynomial in the multivariate polynomial ring~$\F_{2^n}[X_1,\dots, X_r]$ as a polynomial of the form
	\[
		P(X_1,\dots,X_r) = \sum_{j=1}^{r} \left(\sum_{i\ge 0}\alpha_{i,j} X_j^{2^i}\right).
	\]

We will use such polynomials to study the equivalence of APN functions. In this paper, we are interested in \emph{inequivalent} APN functions. There are several notions of equivalence between vectorial Boolean functions that preserve the APN property. We list them in the following definition.

\begin{definition}
	\label{def:equivalence}
	Two functions $f,g \colon \F_{2^n} \to \F_{2^n}$ are called 
	\begin{itemize}
		\item \emph{Carlet-Charpin-Zinoviev equivalent} (CCZ-equivalent), if there is an affine permutation $C$ on $\F_{2^n} \times \F_{2^n}$ such that
		\[
			C(G_f) = G_g,
		\]
		where $G_f = \{(x,f(x)) : x \in \F_{2^n}\}$ is the graph of $f$,
		\item \emph{extended affine equivalent} (EA-equivalent) if there exist three affine functions $A_1,A_2,A_3 \colon \F_{2^n} \to \F_{2^n}$, where $A_1$ and $A_2$ are permutations, such that
		\[
			f(A_1(x)) = A_2(g(x)) + A_3(x),
		\]
		\item \emph{extended linearly equivalent} (EL-equivalent) if they are EA-equivalent and $A_1,A_2$ and $A_3$ are linear,
		\item \emph{affine equivalent} if they are EA-equivalent and $A_3(x) = 0$,
		\item \emph{linearly equivalent} if they are EL-equivalent and $A_3(x)=0$.
	\end{itemize}
\end{definition}
	
	In the case of EL- or linear equivalence, we usually write $L,N,M$ instead of $A_1,A_2,A_3$ to underline the linearity of these functions.
	CCZ-equivalence is the most general known notion of equivalence that preserves the APN property. Obviously, linear equivalence implies affine equivalence, and affine equivalence implies EA-equivalence. Similarly, linear equivalence implies EL-equivalence, and EL-equivalence implies EA-equivalence. Moreover, it is well known that EA-equivalence implies CCZ-equivalence but, in general, the converse is not true. For quadratic APN~functions, however, \textcite{yoshiara2012} proved that also the converse holds.
	
	\begin{proposition}[{\textcite[Theorem~1]{yoshiara2012}}]
	\label{prop:yoshiara}
		Let $f$ and $g$ be quadratic APN functions on a finite field $\F_{2^n}$ with $n \ge 2$. Then $f$ is CCZ-equivalent to $g$ if and only if $f$ is EA-equivalent to $g$.
	\end{proposition}
	
	In this paper, \autoref{prop:yoshiara} will allow us to prove the CCZ-inequivalence of certain quadratic APN functions by showing that they are EA-inequivalent. \par
	
	We characterize some of the mappings that define an equivalence of two functions in the sense of \autoref{def:equivalence} in more detail. Let $f,g$ be functions on $\F_{2^n}$, and denote their graphs by $G_f$ and $G_g$, respectively. We call an affine permutation~$C$ on $\F_{2^n} \times \F_{2^n}$ such that $C(G_f) = G_g$ a \emph{CCZ-mapping} from $g$ to $f$. Similarly to \textcite{canteaut2019a}, we define an \emph{EL-mapping $C_{EL} =(L,M,N)$} from $g$ to $f$ as a linear CCZ-mapping from $g$ to $f$ satisfying
	\[
		f(L(x)) = N(g(x)) + M(x),
	\]
	where $L,N$ are linear permutations on $\F_{2^n}$ and $M$ is a linear map on $\F_{2^n}$. Such an EL-mapping $C_{EL}$ from $g$ to $f$ may be represented as a formal matrix
	\[
		C_{EL} = \begin{bmatrix}L &0\\ M& N\end{bmatrix}
	\]
	corresponding to the calculation
	\[ 
	\begin{bmatrix}	L & 0\\M & N \end{bmatrix}
	\begin{bmatrix}	x\\	g(x)\end{bmatrix}
	=
	\begin{bmatrix} L(x)\\	N(g(x))+M(x)\end{bmatrix} 
	=
	\begin{bmatrix}	y\\	f(y) \end{bmatrix}.
	\]
	We moreover define an \emph{EA-mapping $C_{EA} = (L,M,N,a,b)$} from $g$ to $f$ as a CCZ-mapping from $g$ to $f$ whose linear part is an EL-mapping. It is characterized by $L,M,N$ as above and two elements $a,b \in \F_2^n$ such that
	\begin{equation}
	\label{eq:EA-linear_1}
	f(L(x) + a) = N(g(x)) + M(x) + b.
	\end{equation}
	
	Of particular interest are equivalence mappings from $f$ to $f$, that are mappings preserving the graph of $f$.
	
	\begin{definition}
		For a function $f: \F_{2^n} \to \F_{2^n}$ with graph $G_f$, we call an affine permutation~$\A$ on $\F_{2^n} \times \F_{2^n}$ with $\A(G_f) = G_f$ an \emph{automorphism} of $f$. We denote the set of all such mappings by $\Aut(f)$. If $\A$ is an EA-mapping, we say that $\A$ is an \emph{EA-automorphism} of $f$, and we denote the set of all EA-automorphisms by $\Aut_{EA}(f)$. Analogously, if $\A$ is an EL-mapping, we say that $\A$ is an \emph{EL-automorphism} of $f$, and we denote the set of all EL-automorphisms by $\Aut_{EL}(f)$.
	\end{definition}
		Note that $\Aut(f), \Aut_{EA}(f)$ and $\Aut_{EL}(f)$ each form a group under composition, see \textcite{canteaut2019a}, and $\Aut_{EL}(f)$ is a subgroup of $\Aut_{EA}(f)$, which in turn is a subgroup of $\Aut(f)$. Hence, we simply call $\Aut(f)$ the \emph{automorphism group} of $f$, and we call $\Aut_{EA}(f)$ and $\Aut_{EL}(f)$ the automorphism group of $f$ under EA- or EL-equivalence, respectively.
	
	All the functions we study in this paper are quadratic and have no constant term. We show that if any two such functions $f$ and $g$ are EA-equivalent, they are also EL-equivalent.	
	\begin{proposition}
	\label{prop:EA-EL}
	Suppose $f$ and $g$ are EA-equivalent quadratic functions on $\F_2^n$ with $f(0) = g(0) = 0$, and denote by $C_{EA}=(L,M,N,a,b)$ an EA-mapping from $g$ to $f$. Define a mapping $D_{f,L,a}$ on $\F_2^n$ as
	\[
		D_{f,L,a}(x) = f(L(x)+a) + f(L(x)) + f(a).
	\]
	Then $b = f(a)$, the functions $f$ and $g$ are EL-equivalent, and $C_{EA}$ uniquely defines an EL-mapping $C_{EL} = (L,M+D_{f,L,a},N)$ from $g$ to $f$.
	\end{proposition}
	\begin{proof}
	Recall that $C_{EA}$ satisfies \cref{eq:EA-linear_1}. As $f$ is quadratic, it is easy to confirm that $D_{f,L,a}$ is linear for $a \ne 0$ and zero for $a = 0$. Combining \cref{eq:EA-linear_1} with the definition of $D_{f,L,a}$, we obtain
	\[
		f(L(x)) = N(g(x)) + M(x) + D_{f,L,a}(x) + b + f(a).
	\]
	As $f(0) = g(0) = 0$ and $L,N,M,D_{f,L,a}$ have no constant part either, it follows that $b = f(a)$, which implies
	\[
		f(L(x)) = N(g(x)) + M(x) + D_{f,L,a}(x).
	\]
	Consequently, $C_{EA}$ corresponds to an EL-mapping $C_{EL}$ from $g$ to $f$ of the shape 
	\[
		\begin{bmatrix}L&0\\M + D_{f,L,a}&N	\end{bmatrix},
	\]
	that is uniquely determined by $C_{EA}$.
	\end{proof}
	
	With the help of \autoref{prop:EA-EL}, we can also establish a connection between the automorphism groups $\Aut_{EA}(f)$ and $\Aut_{EL}(f)$ of a quadratic function $f$ with no constant part. \autoref{prop:AutEA_AutEL} may be well-known. We need the definition of a semidirect product first. Let~$G$ be a group with identity element~$e$. Let $H$ and $N$ be two subgroups of $G$. If~$N$ is normal, $G=NH$ and $N\cap H=\{e\}$, then we say $G$ is a \emph{semidirect product} of $N$ and $H$ and write
	\[
		G= N\rtimes H.
	\]
	
	\begin{proposition}
	\label{prop:AutEA_AutEL}
		Let $f$ be a quadratic function on $\F_{2^n}$ with $f(0) = 0$. Then
		\[
			\Aut_{EA}(f) = T_f \rtimes \Aut_{EL}(f),
		\]
		where $T_f$ is isomorphic to the additive group $(\F_{2^n},+)$ of $\F_{2^n}$.
	\end{proposition}
	
	\begin{proof}
		By \autoref{prop:EA-EL}, every EA-automorphism of $f$ given by $(L,M,N,a,b)$ can be uniquely written as the composition of an EL-automorphism $\varphi$ of the shape 
		\begin{equation}\label{eq:varphi}
			\varphi: \begin{bmatrix}
				x\\
				y
			\end{bmatrix}
			\mapsto
			\begin{bmatrix}
				L & 0\\
				\tilde{M} & N
			\end{bmatrix}
			\begin{bmatrix}
				x\\
				y
			\end{bmatrix},
		\end{equation}
		where $\tilde{M}=M + D_{f,L,a}$ for $D_{f,L,a}$ as defined in \autoref{prop:EA-EL}, and a map $\tau_a$ of the shape
		\[
			\tau_a:
			\begin{bmatrix}
				x\\
				y
			\end{bmatrix}
			\mapsto 
			\begin{bmatrix}
				I & 0\\
				D_{f,I,a} & I
			\end{bmatrix}
			\begin{bmatrix}
				x\\
				y
			\end{bmatrix}
			+
			\begin{bmatrix}
				a\\
				f(a)
			\end{bmatrix},
		\]
		where $I$ is the identity map on $\F_{2^n}$.\par
		
		Note that the set of all $\varphi$ is $\Aut_{EL}(f)$. Clearly, $\tau_a$ is also an EA-automorphism of $f$ mapping $(x,f(x))$ to $(x+a, f(x+a))$ for any $x\in \F_{2^n}$. The set of all $\tau_a$ with $a\in \F_{2^n}$ forms a subgroup $T_f$ of $\Aut_{EA}(f)$, which is isomorphic to $(\F_{2^n}, +)$. Hence $\Aut_{EA}(f) = T_f \Aut_{EL}(f)$. Moreover, it is obvious that the identity map on $\F_{2^n} \times \F_{2^n}$ is the unique common element of $T_f$ and $\Aut_{EL}(f)$\par
		
		It remains to show that $T_f$ is a normal subgroup of $\Aut_{EA}(f)$. We do so by verifying that
		\begin{equation}
		\label{eq:Tf_normal}
			\tau_{a} \circ \varphi=\varphi \circ \tau_{L^{-1}(a)}.
		\end{equation}
		A similar result was given by \textcite[Lemma~2.5]{dempwolff2014}. The left-hand side of~\cref{eq:Tf_normal}, $\tau_a\circ \varphi$, is exactly the EA-automorphism $(L,M,N,a,b)$ we decomposed above. The right-hand side of~\cref{eq:Tf_normal}, $\varphi \circ \tau_{L^{-1}(a)}$, maps $(x,f(x))$ to 
		\begin{equation}
		\label{eq:normal_tau}
			\begin{split}
				\varphi \circ \tau_{L^{-1}(a)}
				\begin{bmatrix}
					x\\f(x)
				\end{bmatrix}
				&=
				\begin{bmatrix}
					L & 0\\
					\tilde{M} & N
				\end{bmatrix}
				\begin{bmatrix}
					x+L^{-1}(a)\\
					f(x+L^{-1}(a))
				\end{bmatrix}\\
				&=
				\begin{bmatrix}
					L(x)+a\\N(f(x+L^{-1}(a))) +\tilde{M}(x+L^{-1}(a))
				\end{bmatrix}.	
			\end{split}
		\end{equation}
		
		We consider
		\begin{equation}
		\label{eq:normal_tau_2}
			N(f(x+L^{-1}(a))) +\tilde{M}(x+L^{-1}(a)).
		\end{equation}
		Adding $N(f(x)) + N(f(L^{-1}(a)))$ twice and using the definition of $\tilde{M}$, \cref{eq:normal_tau_2} equals
		\begin{align*}
			&N(f(x)) + M(x) +N(f(L^{-1}(a))+ M(L^{-1}(a))\\
			&\qquad+D_{f,L,a}(x)  +N(f(x+L^{-1}(a)))+N(f(x))+N(f(L^{-1}(a))) +D_{f,L,a}(L^{-1}(a)).
		\end{align*}
		First, note that $D_{f,L,a}(L^{-1}(a))=0$. Second, as we have $f(L(x)) = N(f(x)) +M(x) +D_{f,L,a}(x)$ by the definition of $\varphi$, it follows that
		\begin{align*}
			&N(f(x+L^{-1}(a))) + N(f(x)) + N(f(L^{-1}(a)))\\
			&\qquad\qquad\qquad = f(L(x)+a) + f(L(x)) + f(a) = D_{f,L,a}(x).
		\end{align*}
		Third, using the same reasoning as before and recalling that $D_{f,L,a}(L^{-1}(a))=0$, we have
		\[
		N(f(L^{-1}(a))+ M(L^{-1}(a))=f(a)+D_{f,L,a}(L^{-1}(a))=f(a).
		\]
		Consequently, we obtain
		\[
			N(f(x+L^{-1}(a))) +\tilde{M}(x+L^{-1}(a)) = N(f(x)) + M(x) + f(a),
		\]
		which, considering \cref{eq:normal_tau}, means that $\varphi \circ \tau_{L^{-1}(a)}$ also describes the EA-automorphism $(L,M,N,a,b)$. Therefore, by definition, $\Aut_{EA}(f)= T_f \rtimes \Aut_{EL}(f)$.
	\end{proof}
	
	We remark that \autoref{prop:AutEA_AutEL} enables us to determine the automorphism group $\Aut_{EA}(f)$ under EA-equivalence of any quadratic function $f$ on $\F_{2^n}$, also if $f(0) \ne 0$. To obtain $\Aut_{EA}(f)$, we only have to apply a conjugation of translation on $\Aut_{EA}(f+f(0))$, which we can determine with \autoref{prop:AutEA_AutEL}.\par

	Regarding the automorphism groups of quadratic APN functions, we may say even more: the following lemma follows from \citeauthor{yoshiara2012}'s~\cite{yoshiara2012} proof of \autoref{prop:yoshiara} in combination with a result by \textcite[Theorem~4.10]{dempwolff2014}.
	
	\begin{lemma}
		\label{lem:CCZ_EA}
		Let $f$ be a quadratic APN function on the finite field $\F_{2^n}$, where $n \ge 4$. Then
		\[
		\Aut(f) = \Aut_{EA}(f).
		\]
	\end{lemma}

	We close this section by introducing the general framework we use to study the equivalence of functions in the remainder of this paper. We will mostly consider functions on vector spaces of even dimension~$n = 2m$. Such functions can be represented in a \emph{bivariate} description as a map on $\Fpm^2 = \Fpm \times \Fpm$ with two coordinate functions. As all the functions we study are quadratic and have no constant term, we may use \autoref{prop:yoshiara} in combination with \autoref{prop:EA-EL} to study their CCZ-equivalence by focusing on EL-mappings. We describe EL-equivalence as follows: Two functions $f,g \colon \Fpm^2 \to \Fpm^2$, where 
	\begin{align*}
		f(x,y) = (f_1(x,y), f_2(x,y)) &&\text{and} &&g(x,y) = (g_1(x,y), g_2(x,y))
	\end{align*}
	for coordinate functions $f_1,f_2,g_1,g_2 \colon \Fpm^2 \to \Fpm$, are EL-equivalent, if there exist linear functions $L,N,M \colon \Fpm^2 \to \Fpm^2$, where $L$ and $N$ are bijective, such that
	\[
		f(L(x,y)) = N(g(x,y)) + M(x,y).
	\]
	Write 
	\begin{align*}
		L(x,y) = (L_A(x,y), L_B(x,y)) && \text{and} && M(x,y) = (M_A(x,y), M_B(x,y))
	\end{align*}
	for linear functions $L_A, L_B, M_A, M_B \colon \Fpm^2 \to \Fpm$ and 
	\[
		N(x,y) = \left(N_1(x) + N_3(y),\ N_2(x) + N_4(y)\right)
	\] 
	for linear functions $N_1, \dots, N_4 \colon \Fpm \to \Fpm$. In terms of these newly defined functions, $f$~and $g$ are EL-equivalent if both
	\begin{align}
	\label{eq:LinEquiv_1}
		f_1(L_A(x,y), L_B(x,y)) &= N_1(g_1(x,y)) + N_3(g_2(x,y)) + M_A(x,y),\\
	\label{eq:LinEquiv_2}
		f_2(L_A(x,y), L_B(x,y)) &= N_2(g_1(x,y)) + N_4(g_2(x,y)) + M_B(x,y)
	\end{align}
	hold. They are linearly equivalent if $M(x,y) = 0$.\par
	
	Equations \cref{eq:LinEquiv_1} and \cref{eq:LinEquiv_2} will form the general framework in the proof of our main theorem.

\section{Known classes of APN functions}
\label{sec:known_classes}
	In this section, we give a short overview over the currently known APN functions. In \autoref{tab:powerfunctions}, we present the known APN power functions. 
	\begin{table}[b]
		\centering
		\caption{List of known APN power functions $x \mapsto x^d$ \cite[Table~3]{pott2016}.}
		\label{tab:powerfunctions}
		\begin{tabularx}{\textwidth}{XXXl}
			\hline
			&Exponents $d$	&Conditions	&Reference\rule[-.5em]{0em}{1.5em}\\\hline
			Gold functions		&$2^i+1$		&$\gcd(i,n) = 1,\ i \le \lfloor \frac{n}{2} \rfloor$	&\cite{gold1968,nyberg1994}\rule[-.5em]{0em}{1.5em}\\
			Kasami functions		&$2^{2i}-2^i+1$	&$\gcd(i,n) = 1,\ i \le \lfloor \frac{n}{2} \rfloor$	&\cite{janwa1993,kasami1971}\rule[-.5em]{0em}{0em}\\
			Welch function		&$2^k+3$		&$n = 2k+1$	&\cite{dobbertin1999_welch}\rule[-.5em]{0em}{0em}\\
			Niho function		&$2^k+2^{\frac{k}{2}}-1$, $k$ even	& $n=2k+1$	&\cite{dobbertin1999_niho}\rule[-.5em]{0em}{0em}\\
			&$2^k+2^{\frac{3k+1}{2}}-1$, $k$ odd	&	$n=2k+1$	&\rule[-.5em]{0em}{0em}\\
			Inverse function		&$2^{2k}-1$		&$n=2k+1$	&\cite{beth1994,nyberg1994}\rule[-.5em]{0em}{0em}\\
			Dobbertin function	&$2^{4k} + 2^{3k} + 2^{2k} + 2^{k}-1$	&$n=5k$	&\cite{dobbertin2001}\rule[-.5em]{0em}{0em}\\\hline
		\end{tabularx}
	\end{table}
	This list is conjectured to be complete. APN power functions and their equivalence relations are very well studied. It is well known that the classes in \autoref{tab:powerfunctions} are in general CCZ-inequivalent. Moreover, it is, for example, known that Gold functions are inequivalent for different values of $i$; see \textcite{budaghyancarletleander2008}. \par
	
	As far as non-power APN functions are concerned, the situation becomes much less clear than for power functions. Several infinite families of non-power APN functions have been found, but only for few of them their equivalence relations are known. This includes equivalence relations both between functions from different classes as well as between functions coming from the same class. A current list of known families of APN functions that are CCZ-inequivalent to power functions was recently given by \textcite[Table~3]{budaghyan2020}. This list contains 13 distinct classes, all of which are quadratic.\par
	
	In the present paper, we study the family~(F12) from this list. It was introduced by \textcite{taniguchi2019} who used a criterion developed by \textcite{carlet2011} to prove the APN property of his functions. In \autoref{th:TaniguchiAPN}, we restate Taniguchi's~\cite{taniguchi2019} construction in bivariate representation. Its univariate form can be found in the list by \textcite{budaghyan2020}. 
	\begin{theorem}[{\cite[Theorem~3]{taniguchi2019}}]
	\label{th:TaniguchiAPN}
		Let $m \ge 2$ and $k$ be positive integers such that $\gcd(k,m) = 1$. Let $\alpha, \beta \in \Fpm$ and $\beta \ne 0$. Then the function $f_{k,\alpha,\beta}: \Fptwom \to \Fptwom$, where
		\[
			f_{k,\alpha,\beta}(x,y) = \left(x^{2^{2k}(2^k+1)} + \alpha x^{2^{2k}} y^{2^k} + \beta y^{2^k+1},\ xy\right)
		\]
		is APN if and only if the polynomial $X^{2^k+1} + \alpha X + \beta \in \Fpm[X]$ has no root.
	\end{theorem}
	
	We remark that the Taniguchi APN functions from \autoref{th:TaniguchiAPN} are quadratic. In the following lemma we specify the case $\alpha = 0$.
	
	\begin{lemma}
	\label{lem:alpha=0}
		A Taniguchi function $f_{k,0,\beta}$ on $\Fptwom$ is APN if and only if $m$ is even and $\beta$ is a non-cube in $\Fpm^*$.
	\end{lemma}
	\begin{proof}
		According to \autoref{th:TaniguchiAPN}, the function $f_{k,0,\beta}$ is APN if and only if the polynomial $P(X) \in \Fpm[X]$, where $P(X) = X^{2^k+1} + \beta$, has no root. Recall that $m$ and $k$ are coprime. Hence,
		\[
			\gcd(2^k+1,2^m-1) =
			\begin{cases}
			1,	&\text{if $m$ is odd},\\
			3,	&\text{if $m$ is even}.
			\end{cases}
		\]
		Consequently, if $m$ is odd, $P(X)$ is a permutation polynomial and, thus, always has a root. If $m$ is even, however, then $P(X)$ has a root if and only if $\beta$ is a cube.
	\end{proof}
	
	The following lemma provides insight on the total number of Taniguchi APN functions for given $m$ and $k$---without considering equivalence---by giving the number of admissible~$\beta \in \Fpm^*$. This result is due to \textcite[Theorem~5.6]{bluher2004} who proved it in a more general setting. In the specific form of the present paper, the result was also obtained by \textcite{hellesethkholosha2008}.
	
	\begin{lemma}
	\label{lem:number_of_beta}
		Let $k,m$ be coprime integers such that $0 < k < m$. The number of $\beta \in \Fpm^*$ such that the polynomial $X^{2^k+1} + X + \beta$ has no roots in $\Fpm$ is $\frac{2^m-1}{3}$ if $m$ is even and $\frac{2^m+1}{3}$ if $m$ is odd.
	\end{lemma}
	
	In \autoref{th:ZhouPottAPN}, we present another family of APN functions, which is closely related to \citeauthor{taniguchi2019}'s~\cite{taniguchi2019} construction from \autoref{th:TaniguchiAPN}. It was introduced by Pott and the second author~\cite{zhou2013}, and \textcite{anbar2019} showed that the conditions on the parameters are not only sufficient but also necessary. The equivalence problem of these APN functions was recently solved by the present authors~\cite{kasperszhou2020}.
	
	\begin{theorem}[{\cite[Corollary~2]{zhou2013} and \cite[Proposition~3.5]{anbar2019}}]
		\label{th:ZhouPottAPN}
		Let $m$ be an even integer and let $k,s$ be integers, $0 \le k,s \le m$, such that $\gcd(k,m)=1$. Let $\alpha \in \Fpm^*$. The function $g_{k,s,\alpha}:\Fptwom \to \Fptwom$ defined as
		\[
			g_{k,s,\alpha}(x,y) = \left(x^{2^k+1} + \alpha y^{2^s(2^k+1)},\ xy\right)
		\]
		is APN if and only if $s$ is even and $\alpha$ is a non-cube.
	\end{theorem}

	In the following \autoref{lem:ZP_trivial_equivalence} and \autoref{th:ZP_equivalence}, we restate two results by the present authors~\cite{kasperszhou2020} about the equivalence between Pott-Zhou APN functions that we will need to study the equivalence relations between Taniguchi APN functions in \autoref{sec:Taniguchi_equivalence}.
	
	\begin{lemma}[{\cite[Lemma~5.1]{kasperszhou2020}}]
		\label{lem:ZP_trivial_equivalence}
		Let $m \ge 2$ be an even integer. Let $k,\ell$ be integers coprime to~$m$ such that $0 < k,\ell < m$, and let $s,t$ be even integers with $0 \le s,t \le m$. Let $\alpha, \alpha' \in \Fpm^*$ be non-cubes. The two APN functions $g_{k,s,\alpha},g_{\ell,t,\alpha'}$ on $\Fptwom$ from \autoref{th:ZhouPottAPN} are linearly equivalent
		\begin{enumerate}[label=(\alph*),ref=(\alph*)]
			\item\label{item:non-cubics} if $k = \ell$ and $s=t$, no matter which non-cubes $\alpha$ and $\alpha'$ we choose,
			\item\label{item:k=-l,s=-t} if $k \equiv \pm \ell \pmod{m}$ and $s \equiv \pm t\pmod{m}$.
		\end{enumerate}
	\end{lemma}
	
	\begin{theorem}[{\cite[Theorem~1.1]{kasperszhou2020}}]
		\label{th:ZP_equivalence}
		Let $m \ge 4$ be an even integer. Let $k,\ell$ be integers coprime to $m$ such that $0 < k,\ell < \frac{m}{2}$, let $s,t$ be even integers with $0 \le s,t \le \frac{m}{2}$, and let $\alpha,\alpha' \in \Fpm^*$ be non-cubes. Two Pott-Zhou APN~functions $g_{k,s,\alpha}, g_{\ell,t,\alpha'}$ on $\Fptwom$ from \autoref{th:ZhouPottAPN}, are CCZ-equivalent if and only if $k = \ell$ and $s = t$.
	\end{theorem}

\section{On the equivalence of Taniguchi APN functions}
\label{sec:Taniguchi_equivalence}
	In this section, we study the equivalence problem of the Taniguchi APN functions on~$\Fptwom$, which were introduced in \autoref{th:ZhouPottAPN}. We will answer the question for which values of the parameters $k,\alpha,\beta$ two Taniguchi APN functions $f_{k,\alpha,\beta}$ are CCZ-inequivalent.\par
	
	As we have pointed out before, Taniguchi APN functions are quadratic. Hence, by \autoref{prop:yoshiara}, two Taniguchi APN functions are CCZ-equivalent if and only if they are EA-equivalent, and their automorphism groups under CCZ- and EA-equivalence are the same. We begin by studying the case $\alpha = 0$. Recall from \autoref{lem:alpha=0} that $f_{k,0,\beta}$ is APN if and only if $m$ is even and $\beta$ is a non-cube.
	
	\begin{proposition}
		\label{prop:alpha=0_ZhouPott}
		Let $m \ge 2$ be an even integer, and let $0 < k < \frac{m}{2}$ such that $k$ and $m$ are coprime. Let $\beta, \gamma \in \Fpm^*$ be non-cubes. The Taniguchi APN function $f_{k,0,\beta}$ on $\Fptwom$ from \autoref{th:TaniguchiAPN} is linearly equivalent to the Pott-Zhou APN function $g_{k,2k,\gamma}$ on $\Fptwom$ from \autoref{th:ZhouPottAPN}.
	\end{proposition}
	\begin{proof}
		If $\beta$ is a non-cube in $\Fptwom^*$, then $\frac{1}{\beta}$ is as well. From \autoref{lem:ZP_trivial_equivalence}~(a), we know that the Pott-Zhou APN function $g_{k,2k,\gamma}$ is linearly equivalent to $g_{k,2k,\frac{1}{\beta}}$. We will show that $f_{k,0,\beta}$ is linearly equivalent to $g_{k,2k,\frac{1}{\beta}}$.\par
		
		By \cref{eq:LinEquiv_1} and \cref{eq:LinEquiv_2} and the explanations below, the two functions $f_{k,0,\beta}$ and $g_{k,2k,\frac{1}{\beta}}$ are linearly equivalent if there exist bijective mappings $L,N$ on $\F_{2^m}^2$, represented by linearized polynomials $L_A(X,Y), L_B(X,Y) \in \Fpm[X,Y]$ and $N_1(X),\dots,N_4(X) \in \Fpm[X]$, respectively, such that the two equations
		\begin{align*}
			L_A(x,y)^{2^{2k}(2^k+1)} + \beta L_B(x,y)^{(2^k+1)} &= N_1(x^{2^k+1} + \tfrac{1}{\beta} y^{2^{2k}(2^k+1)}) + N_3(xy),\\
			L_A(x,y)L_B(x,y) &= N_2(x^{2^k+1}+ \tfrac{1}{\beta} y^{2^{2k}(2^k+1)}) + N_4(xy)
		\end{align*}
		hold for all $x,y \in \Fpm$. The functions $f_{k,0,\beta}$ and $g_{k,2k,\frac{1}{\beta}}$ are linearly equivalent by
		\begin{align*}
			L_A(X,Y) &= Y,& L_B(X,Y)&=X,& N_1(X) &= \beta X,&  N_2(X)=N_3(X)&= 0, &N_4(X) &= X.
		\end{align*}
		Consequently, $f_{k,0,\beta}$ is linearly equivalent to $g_{k,2k,\gamma}$.
	\end{proof}
	
	From \autoref{prop:alpha=0_ZhouPott}, we immediately obtain the following results.
	
	\begin{corollary}
	\label{cor:alpha=0}
		Let $m \ge 4$.	
		\begin{enumerate}[label=(\alph*)]
			\item\label{item:alpha=0_a} Two Taniguchi APN functions $f_{k,0,\beta}$ and $f_{-k,0,\beta}$ on $\Fptwom$ are CCZ-equivalent.
			\item\label{item:alpha=0_b} Two Taniguchi APN functions $f_{k,0,\beta}$ and $f_{\ell,0,\beta'}$ on $\Fptwom$ where $0< k,\ell < \frac{m}{2}$ are CCZ-equivalent if and only if $k = \ell$.
		\end{enumerate}
	\end{corollary}
	\begin{proof}
		Statement \ref{item:alpha=0_a} follows from \autoref{prop:alpha=0_ZhouPott} in combination with \autoref{lem:ZP_trivial_equivalence}~(b). Statement \ref{item:alpha=0_b} follows from \autoref{prop:alpha=0_ZhouPott} in combination with \autoref{th:ZP_equivalence}.
	\end{proof}
	
	We remark that for $m = 2$, all Taniguchi APN functions, no matter if $\alpha$ is zero or not, are CCZ-equivalent to the Gold APN function $x \mapsto x^3$. From now on, we focus on the case $\alpha \ne 0$. In the following \autoref{lem:polynomial_transformation}, we summarize several results about polynomials of the shape $X^{2^k+1} + X + \beta$ that we need to solve the equivalence problem of the Taniguchi APN functions.
	
	\begin{lemma}
	\label{lem:polynomial_transformation}
		Let $m \ge 2$, and let $\alpha, \beta \in \Fpm^*$. The statement the polynomial $X^{2^k+1} + \alpha X + \beta \in \Fpm[X]$ has no roots is equivalent to the following statements:
		\begin{enumerate}[label=(\alph*)]
			\item $X^{2^k+1} + X + \frac{\beta}{\alpha^{2^{-k}+1}} \in \Fpm[X]$ has no roots,
			\item $X^{2^k+1} + X + \beta^{2^i} \in \Fpm[X]$, where $i \in \{0, \dots, m-1\}$, has no roots,
			\item $X^{2^{-k}+1} + X + \beta \in \Fpm[X]$ has no roots.
		\end{enumerate}
	\end{lemma}
	\begin{proof}
		Let $P(X) = X^{2^k+1} + \alpha X + \beta$ such that $P(X)$ has no root in $\Fpm$.
		\begin{enumerate}[label=(\alph*)]
			\item If we substitute $X$ by $\alpha^{2^{-k}}X$ in $P(X)$, we obtain $\alpha^{2^{-k}+1}X^{2^k+1} + \alpha^{2^{-k}+1}X + \beta$.	Factoring out $\alpha^{2^{-k}+1}$ gives the result.
			\item Transform the polynomial $P(X)$ into $X^{2^k+1} + X + \beta^{2^i}$ by applying the automorphism $x \mapsto x^{2^i}$ on the coefficients of $P(X)$.
			\item Let $P'(X) = X^{2^{-k}+1} + X + \beta$. Then $P'(X)$ can be transformed into $P(X)$ by the substitution $X \mapsto (X+1)^{2^k}$.\qedhere
		\end{enumerate}	
	\end{proof}
	
	We now focus on the equivalence relations between Taniguchi APN functions.
	
	\begin{proposition}
	\label{prop:trivial_equivalences}
		Let $m \ge 2$ be an integer. Let $k$ be an integer coprime to $m$ such that $0 < k <m$, and let $\alpha, \beta \in \Fpm^*$. Then, the following pairs of Taniguchi APN functions on $\Fptwom$ from \autoref{th:TaniguchiAPN} are linearly equivalent:
		\begin{enumerate}[label=(\alph*)]
			\item\label{item:alpha_1} $f_{k,\alpha,\beta}$ and $f_{k,1,\frac{\beta}{\alpha^{2^{-k}+1}}}$,
			\item\label{item:beta_frobenius} $f_{k,1,\beta^{2^i}}$ and $f_{k,1,\beta}$ for $i \in \{0,\dots,m-1\}$,
			\item\label{item:-k_k} $f_{-k,1,\beta}$ and $f_{k,1,\beta}$.
		\end{enumerate}
	\end{proposition}
	\begin{proof}
		It follows from \autoref{lem:polynomial_transformation} that all the functions in \autoref{prop:trivial_equivalences} are APN.\par
		
		By \cref{eq:LinEquiv_1} and \cref{eq:LinEquiv_2} and the explanations below, two Taniguchi APN functions $f_{k,\alpha,\beta}$ and $f_{\ell,\alpha',\beta'}$ are linearly equivalent if there exist invertible mappings $L,N$ on $\F_{2^m}^2$, represented by linearized polynomials $L_A(X,Y), L_B(X,Y) \in \Fpm[X,Y]$ and $N_1(X),\dots,N_4(X) \in \Fpm[X]$, respectively, such that the two equations
		\begin{align*}
			L_A(x,y)^{2^{2k}(2^k+1)} + \alpha &L_A(x,y)^{2^{2k}} L_B(x,y)^{2^k} +\beta L_B(x,y)^{(2^k+1)} \\ &= N_1(x^{(2^\ell+1)2^{2\ell}} + \alpha' x^{2^{2\ell}} y^{2^\ell} + \beta' y^{2^\ell+1}) + N_3(xy),\\
			L_A(x,y)L_B(x,y) &= N_2(x^{(2^\ell+1)2^{2\ell}} + \alpha' x^{2^{2\ell}} y^{2^\ell} + \beta' y^{2^\ell+1}) + N_4(xy)
		\end{align*}
		hold for all $x,y \in \Fpm$. We will give such polynomials for (a)--(c). As we have $N_2(X) = N_3(X) = 0$ in all three cases, we will not restate these polynomials in every case.
		\begin{enumerate}[label=(\alph*)]
			\item The functions $f_{k,\alpha,\beta}$ and $f_{k,1,\frac{\beta}{\alpha^{2^{-k}+1}}}$ are linearly equivalent by
			\begin{align*}
				L_A(X,Y) &= X,& L_B(X,Y)&=\tfrac{1}{\alpha^{2^{-k}}}Y,& N_1(X) &=X, &N_4(X) &= \tfrac{1}{\alpha^{2^{-k}}}X.
			\end{align*}
			\item The functions $f_{k,1,\beta^{2i}}$ and $f_{k,1,\beta}$ are linearly equivalent by
			\begin{align*}
				L_A(X,Y) &= X^{2^i},& L_B(X,Y)&=Y^{2^i},& N_1(X) &= X^{2^i},&  N_4(X) &= X^{2^i}.
			\end{align*}
			\item We first show that $f_{-k,1,\beta}$ and $f_{k,\frac{1}{\beta},\frac{1}{\beta}}$ are equivalent. This can be seen choosing
			\begin{align*}
				L_A(X,Y) &= Y^{2^{3k}},& L_B(X,Y)&=X^{2^{3k}},& N_1(X) &= \beta X,&  N_4(X) &= X^{2^{3k}}.
			\end{align*}
			Using \ref{item:alpha_1}, it follows that $f_{k,\frac{1}{\beta},\frac{1}{\beta}}$ is linearly equivalent to $f_{k,1,\beta^{2^{-k}}}$, which, by \ref{item:beta_frobenius}, is linearly equivalent to $f_{k,1,\beta}$.\qedhere
		\end{enumerate}		
	\end{proof}
	
	Next, we present our main theorem. We remark that it only holds for $m \ge 3$ as for $m = 2$, all Taniguchi APN functions are CCZ-equivalent to the Gold APN function~$x \mapsto x^3$. According to \autoref{prop:trivial_equivalences}, for $m \ge 3$, every Taniguchi APN function $f_{k,\alpha,\beta}$, where $\alpha \ne 0$, is linearly equivalent to a Taniguchi APN function $f_{\ell,1,\beta'}$, where $0 < \ell < \frac{m}{2}$. Hence, we will only consider functions $f_{k,1,\beta}$ where $0 < k < \frac{m}{2}$ in our theorem. Note that the structure of the proof of \autoref{th:Taniguchi_equivalence} is similar to the structure of the proof of \autoref{th:ZP_equivalence} by the present authors~\cite{kasperszhou2020}. To keep the paper self-contained we will restate some parts that also appear in \cite{kasperszhou2020}. 
	
	\begin{theorem}[Main Theorem]
		\label{th:Taniguchi_equivalence}
		Let $m \ge 3$ be an integer, and let $k,\ell$ be integers, $0 < k,\ell < \frac{m}{2}$, coprime to $m$. Let $\beta, \beta' \in \Fpm^*$ such that the polynomials $X^{2^k+1} + X + \beta$ and $X^{2^\ell+1} + X + \beta'$ have no roots in $\Fpm$. Two Taniguchi APN functions $f_{k,1,\beta}, f_{\ell,1,\beta'}$ on $\Fptwom$, where
		\[
			f_{k,1,\beta} = (x^{2^{2k}(2^k+1)} + x^{2^{2k}}y^{2^k} + \beta y^{2^k+1},\ xy)
		\]
		and
		\[
			f_{\ell,1,\beta'} = (x^{2^{2\ell}(2^\ell+1)} + x^{2^{2\ell}}y^{2^\ell} + \beta' y^{2^\ell+1},\ xy),
		\]
		are CCZ-equivalent if and only if $k = \ell$ and $\beta' = \beta^{2^i}$ for some $i \in \{0,\dots,m-1\}$.
	\end{theorem}
	\begin{proof}
		We have shown in \autoref{prop:trivial_equivalences} that $f_{k,1,\beta}$ and $f_{k,1,\beta^{2^i}}$ are linearly equivalent and thereby CCZ-equivalent. We will now show the converse: if $f_{k,1,\beta}$ and $f_{\ell,1,\beta'}$ are CCZ-equivalent, then $k=\ell$ and $\beta' = \beta^{2^i}$ for some $i \in \{0,\dots,m-1\}$. \par
		
		For $m = 3$ and $m = 4$, the result can be easily confirmed. If $m=3$, then $k=1$ and, according to \autoref{lem:number_of_beta}, there are three distinct $\beta \in \F_{2^3}^*$ such that $X^3+X+\beta$ has no root in $\F_{2^3}$. Clearly, if $\beta$ meets this condition, then $\beta^2$ and $\beta^4$ do as well. Consequently, for $m = 3$, all three Taniguchi APN functions belong to the same equivalence class. If $m=4$, then $k=1$ and there are five distinct $\beta \in \F_{2^4}^*$ such that $X^3+X+\beta$ has no root, namely $1$ and $\beta, \beta^2, \beta^4, \beta^8$ for some $\beta \ne 1$. Hence, for $m=4$, there exist two equivalence classes: $f_{1,1,1}$ of size $1$ and $f_{1,1,\beta}$, where $\beta \ne 1$, of size $4$. The existence of these two classes was also observed by \textcite{taniguchi2019} who computed the $\Gamma$-ranks for these functions. \par
		
		For the remainder of the proof, let $m\ge 5$. Assume $f_{k,1,\beta}$ and $f_{\ell,1,\beta'}$ are CCZ-equivalent. By \autoref{prop:yoshiara} and \autoref{prop:EA-EL}, this implies that the functions are also EL-equivalent. Hence, analogously to the proof of \autoref{prop:trivial_equivalences}, there exist linearized polynomials $L_A(X,Y), L_B(X,Y), M_A(X,Y), M_B(X,Y) \in \Fpm[X,Y]$ and $N_1(X), \dots, N_4(X) \in \Fpm[X]$, where
		\[
			L(X,Y) = (L_A(X,Y), L_B(X,Y))
		\]
		and
		\[
			N(X,Y) = (N_1(X) + N_3(Y),\ N_2(X) + N_4(Y))
		\]
		are invertible, such that the equations
		\begin{align}
			\begin{split}
			\label{eq:equiv1}
				L_A(x,y)^{2^{2k}(2^k+1)} \,+ &\, L_A(x,y)^{2^{2k}} L_B(x,y)^{2^k} +\beta L_B(x,y)^{2^k+1} \\
				&= N_1(x^{(2^\ell+1)2^{2\ell}} + x^{2^{2\ell}} y^{2^\ell} + \beta' y^{2^\ell+1}) + N_3(xy) + M_A(x,y),
			\end{split}\\
			\label{eq:equiv2}
			L_A(x,y)L_B(x,y) &= N_2(x^{(2^\ell+1)2^{2\ell}} + x^{2^{2\ell}} y^{2^\ell} + \beta' y^{2^\ell+1}) + N_4(xy) + M_B(x,y)
		\end{align}
		hold for all $x,y \in \Fpm$. We write $L_A(X,Y) = L_1(X) + L_3(Y)$ and $L_B(X,Y) = L_2(X) + L_4(Y)$ for linearized polynomials $L_1(X), \dots, L_4(X) \in \F_{2^m}[X]$. Hence,
		\[
			L(X,Y) = \left(L_1(X)+L_3(Y),\ L_2(X) + L_4(Y)\right).
		\]		
		Write 
		\begin{align*}
			L_1(X) = \sum_{i=0}^{m-1} a_i X^{2^i},&& L_2(X) = \sum_{i=0}^{m-1} b_i X^{2^i},&& L_3(Y) = \sum_{i=0}^{m-1} \overline{a}_i Y^{2^i},&& L_4(Y) = \sum_{i=0}^{m-1} \overline{b}_i Y^{2^i}.
		\end{align*}
		Analogously, define linearized polynomials $M_1(X), \dots, M_4(X) \in \F_{2^m}[X]$ such that
		\[
			M(X,Y) = (M_1(X) + M_3(Y), M_2(X) + M_4(Y)).
		\]
		For the remainder of the proof, let $x,y \in \Fpm$. We first prove the following claim.
		
		\paragraph{Claim.} \emph{If $f_{k,1,\beta}$ and $f_{\ell,1,\beta'}$ are EA-equivalent, then $k=\ell$ and each of the linearized polynomials $L_1(X), L_2(X), L_3(Y), L_4(Y)$ is a monomial or zero.}\\
		
		\noindent We will prove the result for $y=0$ and obtain statements for $L_1(X)$ and $L_2(X)$. Using the same approach with $x=0$, identical statements can be obtained for $L_3(Y)$ and $L_4(Y)$. Let $y=0$. Then it follows from \cref{eq:equiv1} and $\cref{eq:equiv2}$ that
		\begin{align}
		\label{eq:y=0_1}
			L_1(x)^{2^{2k}(2^k+1)} + L_1(x)^{2^{2k}} L_2(x)^{2^k} +\beta L_2(x)^{2^k+1} &= N_1(x^{(2^\ell+1)2^{2\ell}}) + M_1(x),\\
		\label{eq:y=0_2}	
			L_1(x)L_2(x) &= N_2(x^{(2^\ell+1)2^{2\ell}}) + M_2(x)
		\end{align}
		for all $x \in \Fpm$. Write 
		\begin{align*}
			N_1(X) = \sum_{i=0}^{m-1} c_i X^{2^i}&& \text{and} &&N_2(X) = \sum_{i=0}^{m-1} d_i X^{2^{i-2\ell}}.
		\end{align*}
		Note that, for convenience, we shift the summation index of $N_2(X)$.\par
		
		As $L(X,Y)$ has to be invertible, it is not possible that both $L_1(X)$ and $L_2(X)$ are zero. First, suppose $L_1(X) \ne 0$ and $L_2(X) = 0$. For the case $L_1(X) = 0$ and $L_2(X) \ne 0$, an identical result can be obtained by symmetry. If $L_1(X) \ne 0$ and $L_2(X) = 0$, then it follows from \cref{eq:y=0_2} that $N_2(X)=M_2(X) = 0$ as the left-hand side is zero, and \cref{eq:y=0_1} becomes
		\begin{equation}
		\label{eq:L2=0-Gold}
			L_1(x)^{2^{2k}(2^k+1)} = N_1(x^{(2^\ell+1)2^{2\ell}}) + M_1(x).
		\end{equation}
		
		From \cref{eq:L2=0-Gold}, it follows that the Gold APN functions $x \mapsto x^{2^k+1}$ and $x \mapsto x^{2^\ell+1}$ on $\Fpm$ have to be EA-equivalent. It was shown by \textcite[Theorem~2.1]{budaghyancarletleander2008} that this implies $k = \ell$. The present authors~\cite[Theorem~4.1]{kasperszhou2020} moreover showed that if $m \ge 5$, the equivalence mappings between equivalent Gold APN functions are linearized monomials. In our case, this means the polynomial $L_1(X)$ is a linearized monomial. In summary, we obtain
		\begin{align}
		\label{eq:L1L2_onezero1}
			L_1(X) &= a_u X^{2^u}& \text{and}&& L_2(X) &= 0
		\end{align}
		for some $u \in \{0, \dots, m-1\}$ and $a_u \in \Fpm^*$. If we consider the case $L_1(X) = 0$ and $L_2(X) \ne 0$, we analogously obtain
		\begin{align}
		\label{eq:L1L2_onezero2}
			L_1(X) &= 0& \text{and}&& L_2(X) &=b_u X^{2^u}
		\end{align}
		for some $u \in \{0, \dots, m-1\}$ and $b_u \in \Fpm^*$. In both cases, $M_1(X)=M_2(X)=0$.\par
		
		Now, let both $L_1(X), L_2(X) \ne 0$. Then \cref{eq:y=0_2} becomes
		\begin{equation}
		\label{eq:y=0_L1L2nonzero}
			\sum_{i=0}^{m-1}a_ib_i x^{2^{i+1}} + \sum_{\substack{i,j=0,\\j \ne i}}^{m-1}a_ib_j x^{2^i+2^j} = \sum_{i=0}^{m-1}d_ix^{(2^\ell+1)2^i} + M_2(x).
		\end{equation}
		Note that the first sum on the left-hand side of \cref{eq:y=0_L1L2nonzero} is linearized. Hence, set $M_2(X) = \sum_{i=0}^{m-1}a_ib_i X^{2^{i+1}}$. We rewrite \cref{eq:y=0_L1L2nonzero} as 
		\[
			\sum_{0 \le i < j \le m-1} (a_ib_j + a_jb_i)x^{2^i+2^j} = \sum_{i=0}^{m-1}d_ix^{2^i+ 2^{i+\ell}}
		\]
		which implies that the equations
		\begin{align}
		\label{eq:Coefficients1}
			a_i b_{i+\ell} + a_{i+\ell} b_i 	&= d_i 	\quad\text{for all } i,\\
		\label{eq:Coefficients2}
			a_i b_j + a_j b_i 					&=0		\quad\ \text{for } j \ne i, i\pm \ell,
		\end{align}
		where the subscripts are calculated modulo $m$, have to hold. We separate the proof into two cases: first, the case that $d_i = 0$ for all $i = 0, \dots, m-1$, and, second, the case that $d_u \ne 0$ for some $u \in \{0,\dots,m-1\}$.
		
		\subparagraph{Case 1.} In this case, we show that if $d_i=0$ for all $i = 0, \dots, m-1$, similarly to~\cref{eq:L2=0-Gold}, the problem can be reduced to the equivalence problem of Gold APN functions that has been studied by the present authors \cite[Theorem~4.1]{kasperszhou2020}. Assume $d_i = 0$ for all $i = 0, \dots, m-1$, which means $N_2(X) = 0$. In this case, \cref{eq:Coefficients1} and \cref{eq:Coefficients2} combine to
		\begin{equation}
		\label{eq:Coefficients3}
			a_i b_j + a_j b_i =0 \quad \text{for } j \ne i.
		\end{equation}
		
		As $L_1(X)$ and $L_2(X)$ are both nonzero, each polynomial has at least one nonzero coefficient. Assume $a_u$ and $b_{u'}$ are nonzero, where $u,u' \in \{0, \dots, m-1\}$. If $u = u'$, the corresponding term in \cref{eq:y=0_2}, that is $a_u b_u X^{2^u+1}$, is linearized and only contributes to $M_2(X)$. If $u \ne u'$, then, by \cref{eq:Coefficients3},
		\[
			a_u b_{u'} + a_{u'} b_u = 0.
		\]
		Consequently, $a_{u'}$ and $b_u$ have to be nonzero as well, and $a_u, a_{u'}, b_u, b_{u'}$ have to meet the condition $\frac{a_u}{b_u} = \frac{a_{u'}}{b_{u'}}$. Define $\Delta = \frac{a_u}{b_u}$ and note that $\Delta \ne 0$. It follows that $(a_j, b_j)$ satisfies either
		\begin{align}
		\label{eq:TypeI_TypeII}
			a_j = b_j &= 0	&&\text{or}	&& \frac{a_j}{b_j} = \Delta
		\end{align}
		for all $j = 0, \dots, m-1$. Consequently, $b_j = \delta a_j$, where $\delta = \frac{1}{\Delta}$, for all $j=0, \dots, m-1$, and $L_2(X)$ is a multiple of $L_1(X)$, namely 
		\begin{equation}
		\label{eq:L2_multiple_L1}
			L_2(X) = \delta L_1(X).
		\end{equation}
		
		We plug $L_1(X)$ and $L_2(X)$ into \cref{eq:y=0_1} and obtain
		\begin{equation}
		\label{eq:PolynomialInL1}
			L_1(x)^{2^{2k}(2^k+1)} + \delta^{2^k} L_1(x)^{2^k(2^k+1)} +  \beta\delta^{2^k+1} L_1(x)^{2^k+1} = N_1(x^{(2^\ell + 1)2^{2\ell}}) + M_1(x).
		\end{equation}
		Define a polynomial $T(X) \in \Fpm[X]$ as
		\[
			T(X) = X^{2^{2k}} + \delta^{2^k} X^{2^k} + \beta \delta^{2^k+1} X
		\]
		and rewrite the left-hand side of \cref{eq:PolynomialInL1} as
		\[
			T(L_1(x)^{2^k+1}).
		\]
		
		We show that $T(X)$ is a permutation polynomial. Since $T(X)$ is linearized, it is sufficient to show that $T(X)$ has no nonzero roots. If $T(X)$ had a nonzero root, it would also be a root of the polynomial
		\[
			T'(X) = X^{2^{2k}-1} + \delta^{2^k}X^{2^k-1} + \beta \delta^{2^k+1}.
		\]
		Substitute $X^{2^k-1}$ by $Z$. Note that this substitution is one-to-one since $\gcd(2^k-1,2^m-1) = 2^{\gcd(k,m)} - 1 = 1$. We obtain
		\[
			T'(Z) = Z^{2^k+1} + \delta^{2^k} Z + \beta \delta^{2^k+1}.
		\]
		By \autoref{lem:polynomial_transformation}, the polynomial $T'(Z)$ has no root if and only if $P(X) = X^{2^k+1} + X + \beta$ has no root. This holds by the definition of $\beta$.\par
		
		Hence, we denote by $T^{-1}(X)$ the inverse of $T(X)$ and rewrite \cref{eq:PolynomialInL1} as
		\begin{equation}
		\label{eq:GoldCase}
			L_1(x)^{2^k+1} = T^{-1}(N_1(x^{(2^\ell+1)2^{2\ell}})) + T^{-1}(M_1(x)).
		\end{equation}
		Since $T^{-1}(X)$ is also linearized, \cref{eq:GoldCase} describes the equivalence problem of two Gold APN functions as in the case that exactly one of $L_1(X)$ and $L_2(X)$ is zero. By \cite[Theorem~4.1]{kasperszhou2020}, it follows that $L_1(X)$ is a monomial. Because of \cref{eq:L2_multiple_L1}, the polynomials $L_1(X)$ and $L_2(X)$ are monomials of the same degree:
		\begin{align}
		\label{eq:L1L2_samedegree}
			L_1(X) &= a_u X^{2^u}& \text{and}&& L_2(X) &= b_u X^{2^u}.
		\end{align}
		Moreover, $M_2(X)=a_u b_u X^{2^{u+1}}$ and $M_1(X)=0$.
		
		\subparagraph{Case 2.} Consider \cref{eq:Coefficients1} and \cref{eq:Coefficients2} again and assume $d_u \ne 0$ for some $u \in \{0,\dots,m-1\}$ which means $N_2(X) \ne 0$. We will show that in this case, similarly to Case~1, the polynomials $L_1(X)$ and $L_2(X)$ need to be monomials. In contrast to Case~1, however, now $L_1(X)$ and $L_2(X)$ will have different degrees.\par
		If $d_u \ne 0$, then, by \cref{eq:Coefficients1}, $a_u$ and $b_u$ cannot be zero at the same time. We will separate the proof of Case~2 into two subcases: first, Case~2.1, where both $a_u$ and $b_u$ are nonzero, and second, Case~2.2, where exactly one of $a_u$ and $b_u$ is nonzero. Both these cases will be separated into several subcases again.
		
		\subparagraph{Case 2.1.}
		Assume $a_u \ne 0$ and $b_u \ne 0$. It follows from \cref{eq:Coefficients2} that all pairs $(a_j,b_j)$, where $j \ne u, u \pm \ell$, satisfy \cref{eq:TypeI_TypeII}. We will first show that the only possible nonzero coefficients are $a_j, b_j$ for $j = u, u \pm \ell, u \pm 2\ell$.\par
		
		By way of contradiction, assume there exists $\ell' \ne 0, \pm\ell, \pm 2\ell$ such that $a_{u+\ell'}$ and $b_{u+\ell'}$ are nonzero. By \cref{eq:TypeI_TypeII}, this implies $\frac{a_{u+\ell'}}{b_{u+\ell'}} = \Delta$. Since $u+\ell' \pm \ell \ne u \pm \ell$, it follows from~\cref{eq:Coefficients1} with $i =u + \ell'$ that both $(a_{u+\ell},b_{u+\ell})$ and $(a_{u-\ell},b_{u-\ell})$ also have to satisfy one of the equations in \cref{eq:TypeI_TypeII}. Hence, \cref{eq:TypeI_TypeII} holds for all $j=0, \dots, m-1$ which means that $L_2(X)$ is a multiple of $L_1(X)$. However, now \cref{eq:y=0_2} implies $N_2(X) = 0$. This is a contradiction.\par
		
		Hence, we assume $a_j=b_j = 0$ for $j \ne u, u \pm \ell, u \pm 2\ell$ for the remainder of Case~2.1. We separate its proof into two subcases, both will lead to contradictions.
		
		\subparagraph{Case 2.1.1.}
		Suppose $a_{u \pm 2\ell}=b_{u \pm 2 \ell}=0$.
		In this case, we obtain only one equation from \cref{eq:Coefficients1}, namely 
		\[
			a_{u-\ell}b_{u+\ell} + a_{u+\ell}b_{u-\ell} = 0.
		\]
		Hence, either
		\begin{enumerate}[label=(\roman*), ref=(\roman*)]
			\item\label{item:Case212_1} $a_{u-\ell} = a_{u+\ell} = 0$ or $b_{u-\ell} = b_{u+\ell} = 0$, meaning that one of $L_1(X)$ and $L_2(X)$ is a monomial and the other one has at most three nonzero coefficients, or
			\item\label{item:Case212_2} $a_{u-\ell} = b_{u-\ell} = 0$ or $a_{u+\ell} = b_{u+\ell} = 0$, meaning that both $L_1(X)$ and $L_2(X)$ have at most two nonzero coefficients, or
			\item\label{item:Case212_3} $a_{u \pm \ell},b_{u \pm \ell} \ne 0$ and $\frac{a_{u-\ell}}{b_{u-\ell}} = \frac{a_{u+\ell}}{b_{u+\ell}}$, meaning that both $L_1(X)$ and $L_2(X)$ are trinomials.
		\end{enumerate}
		We will consider each of these three subcases.
		
		\subparagraph{Subcase~\ref{item:Case212_1}.}
		Assume $b_{u-\ell} = b_{u+\ell} = 0$. The case $a_{u-\ell} = a_{u+\ell} = 0$ follows by symmetry. We consider polynomials
  		\begin{align*}
			L_1(X) &= a_{u-\ell}X^{2^{u-\ell}} + a_u X^{2^u} + a_{u+\ell}X^{2^{u+\ell}} & \text{and}&& L_2(X) &= b_u X^{2^u}
		\end{align*}
		which we plug into the left-hand side of \cref{eq:y=0_1}. We obtain
		\begingroup
		\allowdisplaybreaks
  		\begin{align}
  			\nonumber
  			L_1(x)^{2^{2k}(2^k+1)} &=
			a_{u-\ell}^{2^{2k}(2^k+1)} x^{2^{u-\ell+2k}(2^k+1)} +
			a_{u}^{2^{2k}(2^k+1)} x^{2^{u+2k}(2^k+1)}
			\\&\nonumber\quad+
			a_{u+\ell}^{2^{2k}(2^k+1)} x^{2^{u+\ell+2k}(2^k+1)} +
			a_{u-\ell}^{2^{3k}}a_u^{2^{2k}} x^{2^{u+2k}(2^{k-\ell}+1)}
			\\&\label{eq:Case212_Subcase1_1}\quad+
			a_{u}^{2^{3k}}a_{u+\ell}^{2^{2k}} x^{2^{u+\ell+2k}(2^{k-\ell}+1)} +
			a_{u+\ell}^{2^{3k}}a_{u-\ell}^{2^{2k}} x^{2^{u-\ell+2k}(2^{k+2\ell}+1)}
			\\&\nonumber\quad+
			a_{u-\ell}^{2^{3k}}a_{u+\ell}^{2^{2k}} x^{2^{u+\ell+2k}(2^{k-2\ell}+1)} +
			a_{u}^{2^{3k}}a_{u-\ell}^{2^{2k}} x^{2^{u-\ell+2k}(2^{k+\ell}+1)}
			\\&\nonumber\quad+
			a_{u+\ell}^{2^{3k}}a_{u}^{2^{2k}} x^{2^{u+2k}(2^{k+\ell}+1)}
		\end{align}
		\endgroup
		and
  		\begin{equation}
			\begin{split}
			\label{eq:Case212_Subcase1_2}
				L_1(x)^{2^{2k}} L_2(x)^{2^k} &=
				a_{u-\ell}^{2^{2k}} b_u^{2^k} x^{2^{u+k}(2^{k-\ell}+1)} +
				a_{u}^{2^{2k}} b_u^{2^k} x^{2^{u+k}(2^k+1)}
				\\&\quad+
				a_{u+\ell}^{2^{2k}} b_u^{2^k} x^{2^{u+k}(2^{k+\ell}+1)}
			\end{split}
		\end{equation}
		and
		\begin{equation}
		\label{eq:Case212_Subcase1_3}
			\beta L_2(x)^{2^k+1} = \beta b_u^{2^k+1} x^{2^u(2^k+1)}.
		\end{equation}
		Recall that the right-hand side of \cref{eq:y=0_1} is
		\[
			\sum_{i=0}^{m-1}c_i x^{2^{i+2\ell}(2^\ell+1)} + M_1(x).
		\]
		
		We will show that not all of the first three terms of \cref{eq:Case212_Subcase1_1}, that all contain the factor~$x^{2^k+1}$, can be canceled simultaneously. First, as $0 < \ell < \frac{m}{2}$, the terms cannot cancel each other. Second, if $\ell = \frac{m}{2}-k$, the exponent of $x$ in the sixth term can be written as $2^{u-\frac{m}{2}+2k}(2^k+1)$, but by the same reasoning as above, the sixth term cannot cancel any of the first three terms. Third, if $m$ is odd and $k < \frac{m}{4}$, it is possible that $\ell = 2k$. In this case, the term in~\cref{eq:Case212_Subcase1_3}, the first term of \cref{eq:Case212_Subcase1_2} and the first term of \cref{eq:Case212_Subcase1_1} all contain the factor $x^{2^{u}(2^k+1)}$ and could potentially cancel each other, but the second and third term of \cref{eq:Case212_Subcase1_1} cannot be canceled. Analogously, the third term of \cref{eq:Case212_Subcase1_1} could be canceled if $m$ is odd and $ \frac{m}{4}<k<\frac{m}{2}$ and $\ell = -2k$ but the first and second term would remain. Fourth, if $\ell = k$, the first and the second term of \cref{eq:Case212_Subcase1_1} could be canceled by the second term of~\cref{eq:Case212_Subcase1_2} and the seventh term of \cref{eq:Case212_Subcase1_1}, respectively. However, the third term would remain. In summary, for arbitrary $k$ and $\ell$, the third term of \cref{eq:Case212_Subcase1_1} can never be canceled.\par
		
		We now compare the left-hand side and the right-hand side of \cref{eq:y=0_1}: The summands on the left-hand side that contain the factor $x^{2^i(2^k+1)}$ can only be represented on the right-hand side, if $k = \ell$. Hence, assume $k=\ell$. Now, the fourth and the fifth summand of \cref{eq:Case212_Subcase1_1} as well as the first summand of \cref{eq:Case212_Subcase1_2} become linearized. Consequently,
		\[
			M_1(X) = a_{u-k}^{2^{2k}} b_u^{2^k} X^{2^{u+k+1}} + a_{u-k}^{2^{3k}}a_u^{2^{2k}} X^{2^{u+2k+1}} +
			a_{u}^{2^{3k}}a_{u+k}^{2^{2k}} X^{2^{u+3k+1}}.
		\]
		Next, consider the eighth and the ninth term of \cref{eq:Case212_Subcase1_1} where the eighth term can be summarized with the third term of \cref{eq:Case212_Subcase1_2}:
		\begin{align*}
			a_{u+k}^{2^{3k}}a_{u}^{2^{2k}} x^{2^{u+2k}(2^{2k}+1)},&&
			(a_{u}^{2^{3k}}a_{u-k}^{2^{2k}} + a_{u+k}^{2^{2k}} b_u^{2^k}) x^{2^{u+k}(2^{2k}+1)}.
		\end{align*}
		As $m \ge 5$ and $\gcd(k,m) = 1$, we have $2k \not\equiv \pm k \pmod{m}$. Hence, these terms cannot be represented in the form $c_i x^{2^{i+2k}(2^k+1)}$ on the right-hand side of \cref{eq:y=0_1} which means that their coefficients have to be zero. As $a_u \ne 0$, it follows that $a_{u+k} = 0$ which then implies $a_{u-k} = 0$. Hence, $L_1(X)$ and $L_2(X)$ are monomials of the same degree. As this implies $N_2(X)=0$, it contradicts the assumption of Case~2.
		
		\subparagraph{Subcase~\ref{item:Case212_2}.}
		Assume $a_{u-\ell} = b_{u-\ell} = 0$. The case $a_{u+\ell} = b_{u+\ell} = 0$ follows by symmetry. In our case
  		\begin{align*}
			L_1(X) &= a_u X^{2^u} + a_{u+\ell}X^{2^{u+\ell}} & \text{and}&& L_2(X) &= b_u X^{2^u} + b_{u+\ell}X^{2^{u+\ell}}.
		\end{align*}
		On the left-hand side of \cref{eq:y=0_1}, we obtain
  		\[
	  		\begin{split}
				L_1(x)^{2^{2k}(2^k+1)} &=
				a_{u}^{2^{2k}(2^k+1)} x^{2^{u+2k}(2^k+1)} +
				a_{u+\ell}^{2^{2k}(2^k+1)} x^{2^{u+\ell+2k}(2^k+1)}
				\\&\quad+
				a_{u}^{2^{3k}}a_{u+\ell}^{2^{2k}} x^{2^{u+\ell+2k}(2^{k-\ell}+1)} +
				a_{u+\ell}^{2^{3k}}a_{u}^{2^{2k}} x^{2^{u+2k}(2^{k+\ell}+1)}
			\end{split}
		\]
		and
		\[
	  		\begin{split}
				L_1(x)^{2^{2k}} L_2(x)^{2^k} &=
				a_{u}^{2^{2k}} b_u^{2^k} x^{2^{u+k}(2^k+1)} +
				a_{u+\ell}^{2^{2k}} b_{u+\ell}^{2^k} x^{2^{u+\ell+k}(2^k+1)}
				\\&\quad+
				a_{u}^{2^{2k}} b_{u+\ell}^{2^k} x^{2^{u+\ell+k}(2^{k-\ell}+1)} +
				a_{u+\ell}^{2^{2k}} b_u^{2^k} x^{2^{u+k}(2^{k+\ell}+1)}
			\end{split}
		\]
		and
		\[
			\begin{split}
				\beta L_2(x)^{2^k+1} &= 
				\beta b_u^{2^k+1} x^{2^u(2^k+1)} +
				\beta b_{u+\ell}^{2^k+1} x^{2^{u+\ell}(2^k+1)}
				\\&\quad+
				\beta b_u^{2^k} b_{u+\ell} x^{2^{u+\ell}(2^{k-\ell}+1)} +
				\beta b_{u+\ell}^{2^k} b_{u} x^{2^{u}(2^{k+\ell}+1)}.
			\end{split}
		\]

		By similar reasoning as in Subcase~\ref{item:Case212_1}, not all summands containing the factor $x^{2^k+1}$ can be canceled simultaneously. Consequently, we need $k=\ell$ for these terms to be represented on the right-hand side of \cref{eq:y=0_1}. If $k=\ell$, the following terms, which cannot be canceled, occur on the left-hand side of \cref{eq:y=0_1}:
		\begin{align*}
			a_{u+k}^{2^{3k}}a_{u}^{2^{2k}} x^{2^{u+2k}(2^{2k}+1)},&&
			a_{u+k}^{2^{2k}} b_u^{2^k} x^{2^{u+k}(2^{2k}+1)},&&
			\beta b_{u+k}^{2^{2k}} b_u x^{2^u(2^{2k}+1)}.
		\end{align*}
		As they cannot be represented in the form $c_i x^{2^{i+2k}(2^k+1)}$ on the right-hand side of \cref{eq:y=0_1}, their coefficients need to be zero. Hence $a_{u+k} = b_{u+k} = 0$, which means $L_1(X)$ and $L_2(X)$ are monomials of the same degree. As in Subcase~\ref{item:Case212_1}, this is a contradiction.
		
		\subparagraph{Subcase~\ref{item:Case212_3}.}
		Now,
		\begin{align*}
		L_1(X) &= a_{u-\ell}X^{2^{u-\ell}} + a_u X^{2^u} + a_{u+\ell}X^{2^{u+\ell}}\\ \text{and } L_2(X) &= b_{u-\ell}X^{2^{u-\ell}} + b_u X^{2^u} + b_{u+\ell}X^{2^{u+\ell}},
		\end{align*}
		where all coefficients are nonzero and $\frac{a_{u-\ell}}{b_{u-\ell}} = \frac{a_{u+\ell}}{b_{u+\ell}}$. We plug these polynomials into the left-hand side of \cref{eq:y=0_1}. By similar reasoning as in Subcases~(i) and (ii), not all terms containing the factor $x^{2^k+1}$ can be canceled. Hence, $k=\ell$. Now, the left-hand side contains the following two summands that cannot be canceled:
		\begin{align*}
			a_{u+k}^{2^{3k}}a_{u}^{2^{2k}} x^{2^{u+2k}(2^{2k}+1)},&&
			\beta b_{u}^{2^{2k}} b_{u-k} x^{2^{u-k}(2^{2k}+1)}.
		\end{align*}
		As none of them can be represented on the right-hand side of \cref{eq:y=0_1}, their coefficients need to be zero, which means that $a_{u+k} = b_{u-k} = 0$. This contradicts our assumption.
		
		\subparagraph{Case~2.1.2.}
		Suppose that not all of  $a_{u \pm 2 \ell}, b_{u \pm 2 \ell}$ are zero. Recall that all pairs~$(a_j, b_j)$ where $j \ne u, u\pm \ell$ have to satisfy \cref{eq:TypeI_TypeII}. We consider the case that $a_{u+2\ell}$ and $b_{u+2\ell}$ are nonzero. One can obtain an almost identical result by symmetry when assuming that $a_{u-2\ell}$ and $b_{u-2\ell}$ are nonzero.\par
		
		If $a_{u+2\ell}, b_{u+2\ell} \ne 0$, then, by \cref{eq:TypeI_TypeII}, $\frac{a_{u+2\ell}}{b_{u+2\ell}} = \Delta$. It follows from \cref{eq:Coefficients2} that also $(a_{u-2\ell},b_{u-2\ell})$ and $(a_{u-\ell}, b_{u-\ell})$ have to satisfy \cref{eq:TypeI_TypeII}. However, \cref{eq:Coefficients2} does not provide any restriction on the values of $a_{u+\ell}$ and $b_{u+\ell}$. If $(a_{u+\ell}, b_{u+\ell})$ satisfies \cref{eq:TypeI_TypeII}, then all $(a_j,b_j)$ do and we know from the beginning of Case~2.1 that this implies $N_2(X) = 0$. As before, this is a contradiction. If $(a_{u+\ell}, b_{u+\ell})$ does not satisfy \cref{eq:TypeI_TypeII}, then it follows from \cref{eq:Coefficients2} that $a_j=b_j=0$ for $j =u-\ell, u-2\ell$. Hence, 
		\begin{align*}
		L_1(X) &= a_u X^{2^u} + a_{u+\ell}X^{2^{u+\ell}} + a_{u+2\ell}X^{2^{u+2\ell}}\\
		\text{and } L_2(X) &= b_u X^{2^u} + b_{u+\ell}X^{2^{u+\ell}} + b_{u+2\ell}X^{2^{u+2\ell}}.
		\end{align*}
		
		As $\frac{a_u}{b_u} = \frac{a_{u+2\ell}}{b_{u+2\ell}}$, this case is similar to Case 2.1.1, Subcase~\ref{item:Case212_3}, when we substitute $u$ by $u+\ell$, with the only difference that now, one of the middle coefficients $a_{u+\ell}, b_{u+\ell}$ can be zero. However, the arguments used in the previous case leading to the conclusion $k=\ell$ still hold. If $k = \ell$, the left-hand side of \cref{eq:y=0_1} contains the following terms that cannot be canceled:
		\begin{align*}
			a_{u+2k}^{2^{3k}}a_{u}^{2^{2k}} x^{2^{u+2k}(2^{3k}+1)},&&
			a_{u+2k}^{2^{2k}}b_{u}^{2^k} x^{2^{u+k}(2^{3k}+1)},&&
			\beta b_{u+2k}^{2^k} b_{u} x^{2^u(2^{3k}+1)}.
		\end{align*}		
		They cannot be represented on the right-hand side of \cref{eq:y=0_1}, hence, their coefficients need to be zero. This contradicts our assumption that $a_u, a_{u+2k}, b_u, b_{u+2k}$ are nonzero.
		
		\subparagraph{Case~2.2.}
		Assume, exactly one of $a_u$ and $b_u$ is nonzero. We show the case $a_u \ne 0$ and $b_u = 0$. The case $a_u = 0$ and $b_u \ne 0$ can be proved analogously. So, assume $a_u \ne 0$ and $b_u = 0$. From \cref{eq:Coefficients1} with $i = u$, we obtain the equation
		\[
			a_u b_{u+\ell} = d_u.
		\]
		As $d_u \ne 0$, it follows that $b_{u+\ell} \ne 0$. From \cref{eq:Coefficients2} with $i = u$, we obtain
		\[
			a_u b_j = 0 \quad \text{for } j \ne u, u\pm \ell.
		\]
		Consequently, $b_j = 0$ for $j \ne u \pm \ell$. Now, it follows from \cref{eq:Coefficients2} with $i = u + \ell$ that
		\[
			a_j b_{u+\ell} = 0 \quad \text{for } j \ne u-\ell,u, u + \ell, u + 2\ell.
		\]
		 Consequently, $a_j = 0$ for $j \ne u-\ell,u, u + \ell, u + 2\ell$.	We will separate the proof of Case~2.2 into two subcases: in Case~2.2.1, we consider $b_{u-\ell} \ne 0$, in Case~2.2.2, we suppose $b_{u-\ell} = 0$.
		
		\subparagraph{Case 2.2.1.} Assume $b_{u-\ell} \ne 0$. From \cref{eq:Coefficients2} with $i = u-\ell$ and $j = u+2\ell$, we obtain
		\[
			a_{u+2\ell} b_{u-\ell} = 0,
		\]
		which implies $a_{u+2\ell} = 0$, and
		\[
			a_{u-\ell}b_{u+\ell} + a_{u+\ell} b_{u-\ell} = 0,
		\]
		which, recalling that $b_{u+\ell}$ is nonzero, implies either $a_{u-\ell} = a_{u+\ell} = 0$ or $a_{u-\ell}, a_{u+\ell} \ne 0$ and $\frac{a_{u-\ell}}{b_{u-\ell}} = \frac{a_{u+\ell}}{b_{u+\ell}}$. We separate these two subcases:
		
		\subparagraph{Subcase (i).} Assume $a_{u-\ell} = a_{u+\ell} = 0$. Then
		\begin{align*}
			L_1(X) &= a_u X^{2^u}	&\text{and}	&&L_2(X) &= b_{u-\ell}X^{2^{u-\ell}} + b_{u+\ell}X^{2^{u+\ell}}.
		\end{align*}
		We plug these polynomials into the left-hand side of \cref{eq:y=0_1} and obtain
		\[
			L_1(x)^{2^{2k}(2^k+1)} = a_{u}^{2^{2k}(2^k+1)} x^{2^{u+2k}(2^k+1)}
		\]
		and
		\[
			L_1(x)^{2^{2k}} L_2(x)^{2^k} =
			a_{u}^{2^{2k}} b_{u-\ell}^{2^k} x^{2^{u-\ell+k}(2^{k+\ell}+1)} +
			a_{u}^{2^{2k}} b_{u+\ell}^{2^k} x^{2^{u+\ell+k}(2^{k-\ell}+1)}
		\]
		and
		\begin{equation}
			\begin{split}
			\label{eq:Case221_Subcase1_1}
				\beta L_2(x)^{2^k+1} &= 
				\beta b_{u-\ell}^{2^k+1} x^{2^{u-\ell}(2^k+1)} +
				\beta b_{u+\ell}^{2^k+1} x^{2^{u+\ell}(2^k+1)}
				\\&\quad +
				\beta b_{u-\ell}^{2^k} b_{u+\ell} x^{2^{u+\ell}(2^{k-2\ell}+1)} +
				\beta b_{u+\ell}^{2^k} b_{u-\ell} x^{2^{u-\ell}(2^{k+2\ell}+1)}.
			\end{split}
		\end{equation}
		As in previous cases, if $k \ne \ell$, not all terms containing the factor $x^{2^k+1}$ can be canceled simultaneously. Thus, we need $k = \ell$. However, if $k = \ell$, the left-hand side \cref{eq:y=0_1} contains the term
		\[
			a_u^{2^{2k}}b_{u-k}^{2^k} x^{2^u(2^{2k}+1)}
		\]
		that cannot be represented in the form $c_i x^{2^{i+2k}(2^k+1)}$ on the right-hand side of \cref{eq:y=0_1}. Hence, its coefficient needs to be zero which contradicts our assumption.
		
		\subparagraph{Subcase (ii).} Assume $a_{u-\ell}, a_{u+\ell} \ne 0$ and $\frac{a_{u-\ell}}{b_{u-\ell}} = \frac{a_{u+\ell}}{b_{u+\ell}}$. Then
		\begin{align*}
			L_1(X) &= a_{u-\ell} X^{2^{u-\ell}} + a_u X^{2^u} +a_{u+\ell} X^{2^{u+\ell}}	&\text{and}	&&L_2(X) &= b_{u-\ell}X^{2^{u-\ell}} + b_{u+\ell}X^{2^{u+\ell}}.
		\end{align*}
		We plug these polynomials into the left-hand side of \cref{eq:y=0_1}. Then $L_1(x)^{2^{2k}(2^k+1)}$ is as in~\cref{eq:Case212_Subcase1_1} and $\beta L_2(x)^{2^k+1}$ is as in \cref{eq:Case221_Subcase1_1}. Moreover,
  		\begin{equation}
	  		\begin{split}
		  		\label{eq:Case221_Subcase2_1}
				L_1(x)^{2^{2k}} L_2(x)^{2^k} &=
				a_{u-\ell}^{2^{2k}} b_{u-\ell}^{2^k} x^{2^{u-\ell+k}(2^k+1)} +
				a_{u+\ell}^{2^{2k}} b_{u+\ell}^{2^k} x^{2^{u+\ell+k}(2^k+1)}
				\\&\quad+
				a_{u-\ell}^{2^{2k}} b_{u+\ell}^{2^k} x^{2^{u+\ell+k}(2^{k-2\ell}+1)} +
				a_{u}^{2^{2k}} b_{u-\ell}^{2^k} x^{2^{u-\ell+k}(2^{k+\ell}+1)}
				\\&\quad+
				a_{u}^{2^{2k}} b_{u+\ell}^{2^k} x^{2^{u+\ell+k}(2^{k-\ell}+1)} +
				a_{u+\ell}^{2^{2k}} b_{u-\ell}^{2^k} x^{2^{u-\ell+k}(2^{k+2\ell}+1)}.
	  		\end{split}
		\end{equation}
		By the same reasoning as in Subcase~(i), it follows that $k=\ell$. However, if $k = \ell$, then the fourth term of \cref{eq:Case221_Subcase2_1} cannot be canceled by any other terms on the left-hand side of~\cref{eq:y=0_1}, neither can it be represented on the right-hand side of~\cref{eq:y=0_1}. This implies $b_{u-\ell} = 0$ which contradicts our assumption.
		
		\subparagraph{Case~2.2.2.}
		Assume $b_{u-\ell} = 0$. From \cref{eq:Coefficients2} with $i = u+\ell$ and $j = u-\ell$, it follows that
		\[
			a_{u-\ell}b_{u+\ell} = 0
		\]
		which, recalling that $b_{u+\ell} \ne 0$, implies $a_{u-\ell} = 0$. Then
		\begin{align*}
		L_1(X) &= a_u X^{2^u} + a_{u+\ell} X^{2^{u+\ell}} + a_{u+2\ell} X^{2^{u+2\ell}}	&\text{and}	&&L_2(X) &= b_{u+\ell}X^{2^{u+\ell}}.
		\end{align*}
		Plugging these polynomials into \cref{eq:y=0_1}, the expressions $L_1(x)^{2^{2k}(2^k+1)}$, $L_1(x)^{2^{2k}} L_2(x)^{2^k}$ and $\beta L_2(x)^{2^k+1}$ are as in \cref{eq:Case212_Subcase1_1}, \cref{eq:Case212_Subcase1_2} and \cref{eq:Case212_Subcase1_3}, respectively, where we substitute $u$ by $u+\ell$. By the same reasoning as in Case~2.1.1, Subcase~(i), it follows that $k=\ell$. If $k=\ell$, analogously to Case~2.1.1, Subcase~(i), the following terms occur on the left-hand side of~\cref{eq:y=0_1}:
		\begin{align*}
			a_{u+2k}^{2^{3k}}a_{u}^{2^{2k}} x^{2^{u+2k}(2^{3k}+1)},&&
			(a_{u+k}^{2^{3k}}a_{u}^{2^{2k}} + a_{u+2k}^{2^{2k}} b_{u+k}^{2^k}) x^{2^{u+2k}(2^{2k}+1)}.
		\end{align*}
		As neither of them can be represented on the right-hand side of \cref{eq:y=0_1}, their coefficients need to be zero. As $a_u \ne 0$, it follows that $a_{u+2k} = 0$, and, consequently, $a_{u+k}=0$. Hence, $L_1(X)$ and $L_2(X)$ are monomials of the form
		\begin{align}
		\label{eq:L1L2_differentdegree1}
			L_1(X) &= a_u X^{2^u} &\text{and}	&& L_2(X) = b_{u+k} X^{2^{u+k}},
		\end{align}
		and $M_1(X) = a_u^{2^{2k}}b_{u+k}^{2^k} X^{2^{u+2k+1}}$.\par
		
		Note that if we consider Case 2.2 with $a_u = 0$ and $b_u \ne 0$, we obtain
		\begin{align}
		\label{eq:L1L2_differentdegree2}
			L_1(X) &= a_{u+k} X^{2^{u+k}} &\text{and}	&& L_2(X) = b_{u} X^{2^{u}}
		\end{align}
		and $M_1(X)=a_{u+k}^{2^{2k}} b_u^{2^k} X^{2^{u+2k+1}}$ from Case 2.2.2. This concludes the \textbf{proof of our Claim}.\\
		
		\noindent We summarize the results we have obtained so far. If the Taniguchi APN functions $f_{k,1,\beta}$ and $f_{\ell,1,\beta'}$ are EA-equivalent, then $k = \ell$ and $L_1(X)$ and $L_2(X)$ meet the following conditions: either, one of the polynomials $L_1(X)$ and $L_2(X)$ is zero and the other one is a monomial, see \cref{eq:L1L2_onezero1} and \cref{eq:L1L2_onezero2}, or both $L_1(X)$ and $L_2(X)$ are monomials, either of the same degree or of degrees $u$ and $u+k$, see \cref{eq:L1L2_samedegree}, \cref{eq:L1L2_differentdegree1} and \cref{eq:L1L2_differentdegree2}. Vice versa, the same statements hold for $L_3(Y)$ and $L_4(Y)$.\par
		It remains to be shown that the EA-equivalence of $f_{k,1,\beta}$ and $f_{k,1,\beta'}$ implies $\beta' = \beta^{2^i}$ for some $i \in \{0,\dots,m-1\}$. Combining the results on $L_1(X),L_2(X),L_3(Y),L_4(Y)$ mentioned above, it is clear that the polynomials $L_A(X,Y)$ and $L_B(X,Y)$ have to be of one of the following forms:

		\begin{enumerate}[label=\textbf{(\alph*)}, ref=\textbf{(\alph*)}]
			\item \label{item:a} $L_A(X,Y) = a_u X^{2^u} + \overline{a}_w Y^{2^w}$ and $L_B(X,Y) = b_u X^{2^u} + \overline{b}_w Y^{2^w}$,
			
			\item \label{item:b} $L_A(X,Y) = a_u X^{2^u} + \overline{a}_w Y^{2^w}$ and $L_B(X,Y) = b_u X^{2^u} + \overline{b}_{w+k} Y^{2^{w+k}}$,
			
			\item \label{item:c} $L_A(X,Y) = a_u X^{2^u} + \overline{a}_{w+k} Y^{2^{w+k}}$ and $L_B(X,Y) = b_u X^{2^u} + \overline{b}_w Y^{2^w}$,
			
			\item \label{item:d} $L_A(X,Y) = a_u X^{2^u} + \overline{a}_w Y^{2^w}$ and $L_B(X,Y) = b_{u+k} X^{2^{u+k}} + \overline{b}_w Y^{2^w}$,
			
			\item \label{item:e} $L_A(X,Y) = a_{u+k} X^{2^{u+k}} + \overline{a}_w Y^{2^w}$ and $L_B(X,Y) = b_u X^{2^u} + \overline{b}_w Y^{2^w}$,
			
			\item \label{item:f} $L_A(X,Y) = a_u X^{2^u} + \overline{a}_w Y^{2^w}$ and $L_B(X,Y) = b_{u+k} X^{2^{u+k}} + \overline{b}_{w+k} Y^{2^{w+k}}$,
			
			\item \label{item:g} $L_A(X,Y) = a_u X^{2^u} + \overline{a}_{w+k} Y^{2^{w+k}}$ and $L_B(X,Y) = b_{u+k} X^{2^{u+k}} + \overline{b}_w Y^{2^w}$,
			
			\item \label{item:h} $L_A(X,Y) = a_{u+k} X^{2^{u+k}} + \overline{a}_{w} Y^{2^{w}}$ and $L_B(X,Y) = b_{u} X^{2^{u}} + \overline{b}_{w+k} Y^{2^{w+k}}$,
			
			\item \label{item:i} $L_A(X,Y) = a_{u+k} X^{2^{u+k}} + \overline{a}_{w+k} Y^{2^{w+k}}$ and $L_B(X,Y) = b_{u} X^{2^{u}} + \overline{b}_w Y^{2^w}$.
		\end{enumerate}

		Note that, as $L(X,Y) = (L_A(X,Y), L_B(X,Y))$ has to be a permutation polynomial, it is neither possible that $L_A(X,Y)$ or $L_B(X,Y)$ is zero nor that both $L_A(X,Y)$ and $L_B(X,Y)$ depend only on $X$ or only on $Y$. We will show that all cases listed above lead to the conclusion that $L_A(X,Y)$ and $L_B(X,Y)$ need to be monomials of the same degree of the shape
		\begin{align}
		\label{eq:LALB_monomials}
			L_A(X,Y) &= a_u X^{2^u}& \text{and}&& L_B(X,Y) &= b_u Y^{2^u}.
		\end{align}
		
		We rewrite \cref{eq:equiv1} and \cref{eq:equiv2} considering $k = \ell$:
		\begin{align}
			\begin{split}
			\label{eq:equiv1_k=l}
				L_A(x,y)^{2^{2k}(2^k+1)} + &L_A(x,y)^{2^{2k}} L_B(x,y)^{2^k} +\beta L_B(x,y)^{2^k+1} \\
				&= N_1(x^{2^{2k}(2^k+1)} + x^{2^{2k}} y^{2^k} + \beta' y^{2^k+1}) + N_3(xy) + M_A(x,y),
			\end{split}\\
			\label{eq:equiv2_k=l}
			L_A(x,y)L_B(x,y) &= N_2(x^{2^{2k}(2^k+1)} + x^{2^{2k}} y^{2^k} + \beta' y^{2^k+1}) + N_4(xy) + M_B(x,y).
		\end{align}
		
		We will plug all the possible combinations \ref{item:a}--\ref{item:i} into these equations. We begin with~\ref{item:b}. By proceeding analogously, the cases \ref{item:c}--\ref{item:e} lead to the same result. If we plug the polynomials of \ref{item:b} into the left-hand side of \cref{eq:equiv2_k=l}, we obtain
		\begin{equation}
			\begin{split}
			\label{eq:b_equiv2}
				L_A(x,y)L_B(x,y) &= 
				a_u b_u x^{2^{u+1}} + \overline{a}_w \overline{b}_{w+k} y^{2^w(2^k+1)} 
				\\&\quad+ 
				a_u \overline{b}_{w+k} x^{2^u}y^{2^{w+k}} + \overline{a}_w b_u x^{2^u} y^{2^w}.
			\end{split}
		\end{equation}
		Note that the first term of \cref{eq:b_equiv2} is linearized. As there is no term containing the factor~$x^{2^k+1}$, we need $N_2(X) = 0$ on the right-hand side of \cref{eq:equiv2_k=l}. This implies, first, that the coefficient $\overline{a}_w \overline{b}_{w+k}$ of the second summand of \cref{eq:b_equiv2} has to be zero, and second, that the third and the fourth summand of \cref{eq:b_equiv2} cannot be represented simultaneously on the right-hand side of~\cref{eq:equiv2_k=l}. The coefficient of the second summand of \cref{eq:b_equiv2} is zero if $\overline{a}_{w}$ or $\overline{b}_{w+k}$ is zero. We separate the proof into two cases: \par
		
		\textbf{Case 1.} Assume $\overline{a}_w = 0$. Note that this implies $a_u \ne 0$ and $\overline{b}_{w+k} \ne 0$ as otherwise $L(X,Y)$ would not be a permutation polynomial. If $\overline{a}_w = 0$, then \cref{eq:equiv2_k=l} holds only if $u=w+k$. Set $u = w+k$ and plug $L_A(x,y)$ and $L_B(x,y)$ into the left-hand side of \cref{eq:equiv1_k=l}. We obtain
		\begin{equation}
			\label{eq:b_equiv1_1}
			L_A(x,y)^{2^{2k}(2^k+1)} =
			a_u^{2^{2k}(2^k+1)} x^{2^{u+2k}(2^k+1)} 
		\end{equation}
		and
		\begin{equation}
		\label{eq:b_equiv1_2}
			L_A(x,y)^{2^{2k}} L_B(x,y)^{2^k} =
			a_u^{2^{2k}} b_u^{2^k} x^{2^{u+k}(2^k+1)} + a_u^{2^{2k}} \overline{b}_u^{2^k} x^{2^{u+2k}} y^{2^{u+k}}
		\end{equation}
		and
		\begin{equation}
			\begin{split}
				\label{eq:b_equiv1_3}
				\beta L_B(x,y)^{2^k+1} &= 
				\beta b_u^{2^k+1} x^{2^u(2^k+1)} + \beta \overline{b}_u^{2^k+1} y^{2^u(2^k+1)}
				\\&\quad  + 
				\beta b_u^{2^k} \overline{b}_u x^{2^{u+k}} y^{2^u} + \beta \overline{b}_u^{2^k} b_u x^{2^u} y^{2^{u+k}}.
			\end{split}
		\end{equation}
		The fourth summand of \cref{eq:b_equiv1_3} cannot be canceled by any other summand of \cref{eq:b_equiv1_1}--\cref{eq:b_equiv1_3} and it cannot be represented on the right-hand side of \cref{eq:equiv1_k=l}. As $\beta, \overline{b}_u \ne 0$, it follows that $b_u = 0$. Consequently, $L_A(X,Y)$ and $L_B(X,Y)$ are monomials of the same degree as in~\cref{eq:LALB_monomials}.\par 
		
		\textbf{Case 2.} Assume $\overline{b}_{w+k} = 0$. By the same reasoning as above, this implies $b_u \ne 0$ and $\overline{a}_u \ne 0$. Now, \cref{eq:equiv2_k=l} holds for $u=w$. Set $u = w$ and plug $L_A(x,y)$ and $L_B(x,y)$ into the left-hand side of~\cref{eq:equiv1_k=l}. The summand $L_A(x,y)^{2^{2k}} L_B(x,y)^{2^k}$ contains the term
		\[
			\overline{a}_u^{2^{2k}} b_u^{2^k} x^{2^{u+k}} y^{2^{u+2k}},
		\]
		that has a nonzero coefficient and cannot be canceled by the other terms on the left-hand side of \cref{eq:equiv1_k=l}. However, it cannot be represented on the right-hand side of \cref{eq:equiv1_k=l}. This is a contradiction.\par
		
		We next study \ref{item:f}. By symmetry, the same result also holds for \ref{item:i}. Moreover, an analogous approach gives identical results for \ref{item:g} and \ref{item:h}. If we plug $L_A(X,Y)$ and $L_B(X,Y)$ of \ref{item:f} into \cref{eq:equiv2_k=l}, we obtain
		\begin{equation}
			\begin{split}
			\label{eq:f_equiv2} 
				L_A(x,y)L_B(x,y) &= 
				a_u b_{u+k} x^{2^u(2^k+1)} + \overline{a}_w \overline{b}_{w+k} y^{2^w(2^k+1)} 
				\\&\quad + 
				a_u \overline{b}_{w+k} x^{2^u}y^{2^{w+k}} + \overline{a}_w b_{u+k} x^{2^{u+k}} y^{2^w}.
			\end{split}
		\end{equation}
		If all coefficients are nonzero, we need $u = w+2k$ to represent the first and the second summand of \cref{eq:f_equiv2} on the right-hand side of \cref{eq:equiv2_k=l}. Then, however, the fourth term of \cref{eq:f_equiv2} cannot be represented on the right-hand side of \cref{eq:equiv2_k=l}, which is a contradiction.\par
		
		Now assume one of the coefficients is zero. We show the case $b_{u+k} = 0$. If we assume $a_u = 0$ instead, we end up with the same contradiction as in Case~2 of the study of \ref{item:b}. By symmetry, analogous results can be obtained when assuming $\overline{a}_w = 0$ or $\overline{b}_{w+k} = 0$. If $b_{u+k} = 0$, it follows that $a_{u}$ and $\overline{b}_{w+k}$ are nonzero as otherwise $L(X,Y)$ would not be a permutation polynomial. Moreover, as the first term of \cref{eq:f_equiv2} vanishes, we need $N_2(X)=0$. Then, also the second term of \cref{eq:f_equiv2} cannot be represented on the right-hand side of \cref{eq:equiv2_k=l} and $\overline{a}_w \overline{b}_{w+k}$ has to be zero. As $\overline{b}_{w+k} \ne 0$, we need $\overline{a}_{w} = 0$ for the second coefficient to be zero. Moreover, we need $u=w+k$ to represent the third summand of \cref{eq:f_equiv2} on the right-hand side of \cref{eq:equiv2_k=l}. Consequently, $L_A(X,Y)$ and $L_B(X,Y)$ are monomials as in \cref{eq:LALB_monomials}. \par
		
		Finally, we study \ref{item:a}. If we plug $L_A(X,Y)$ and $L_B(X,Y)$ of \ref{item:a} into \cref{eq:equiv2_k=l}, we obtain
		\begin{equation}
		\label{eq:a_equiv2}
			L_A(x,y)L_B(x,y) = 
			a_u b_{u} x^{2^{u+1}} + \overline{a}_{w} \overline{b}_{w} y^{2^{w+1}} + 
			(a_u \overline{b}_{w} + \overline{a}_w b_u) x^{2^u}y^{2^w}.
		\end{equation}
		We separate two cases: in the first case, the third term of \cref{eq:a_equiv2} vanishes, in the second case, its coefficient is nonzero. \par
		\textbf{Case~1.} We first show, that the third term of \cref{eq:a_equiv1_2} can only vanish if all coefficients are nonzero. Suppose $a_u = 0$. Then $\overline{a}_w b_u$ has to be zero as well. However, this is not possible, as $a_u = 0$ implies that $\overline{a}_w$ and $b_u$ are nonzero. By symmetry, the same result is obtained if we assume that any other coefficient is zero.\par 
		Consequently, assume all coefficients are nonzero and $\frac{a_u}{b_u} = \frac{\overline{a}_w}{\overline{b}_w}$. Then \cref{eq:a_equiv2} does not provide any information, as the left-hand side is a linearized polynomial. We plug $L_A(X,Y)$ and $L_B(X,Y)$ into the left-hand side of \cref{eq:equiv1_k=l} and obtain
		\begin{equation}
			\begin{split}
			\label{eq:a_equiv1_1}
				L_A(x,y)^{2^{2k}(2^k+1)} &=
				a_u^{2^{2k}(2^k+1)} x^{2^{u+2k}(2^k+1)} +
				\overline{a}_{w}^{2^{2k}(2^k+1)} y^{2^{w+2k}(2^k+1)} 
				\\&\quad +
				a_u^{2^{3k}} \overline{a}_w^{2^{2k}} x^{2^{u+3k}} y^{2^{w+2k}} +
				\overline{a}_w^{2^{3k}} a_u^{2^{2k}} x^{2^{u+2k}} y^{2^{w+3k}}
			\end{split}
		\end{equation}
		and
		\begin{equation}
			\begin{split}
			\label{eq:a_equiv1_2}
				L_A(x,y)^{2^{2k}} L_B(x,y)^{2^k} &=
				a_u^{2^{2k}} b_u^{2^k} x^{2^{u+k}(2^k+1)} + 
				\overline{a}_w^{2^{2k}} \overline{b}_w^{2^k} y^{2^{w+k}(2^k+1)} 
				\\& \quad +
				a_u^{2^{2k}} \overline{b}_w^{2^k} x^{2^{u+2k}} y^{2^{w+k}} +
				\overline{a}_w^{2^{2k}} b_u^{2^k} x^{2^{u+k}} y^{2^{w+2k}}
			\end{split}
		\end{equation}
		and
		\begin{equation}
			\begin{split}
			\label{eq:a_equiv1_3}
				\beta L_B(x,y)^{2^k+1} &= 
				\beta b_u^{2^k+1} x^{2^u(2^k+1)} + \beta \overline{b}_w^{2^k+1} y^{2^w(2^k+1)} +
				\\&\quad + 
				\beta b_u^{2^k} \overline{b}_w x^{2^{u+k}} y^{2^w} + \beta \overline{b}_w^{2^k} b_u x^{2^u} y^{2^{w+k}}.
			\end{split}
		\end{equation}
		No matter how we choose $u$ and $w$, the third and the fourth summand of \cref{eq:a_equiv1_1} cannot be canceled by the terms of \cref{eq:a_equiv1_1}--\cref{eq:a_equiv1_3} and they cannot be represented simultaneously on the right-hand side of \cref{eq:equiv1_k=l}. Hence, at least one of the coefficients needs be zero which is a contradiction.
	
		\textbf{Case~2.} Assume $a_u \overline{b}_{w} + \overline{a}_w b_u \ne 0$. As there are no terms on the left-hand side of \cref{eq:equiv2_k=l} containing the factors $x^{2^k+1}$ and $y^{2^k+1}$, it follows that $N_2(X)=0$, and we need $u=w$ to represent the third summand of~\cref{eq:a_equiv2} on the right-hand side of \cref{eq:equiv2_k=l}. We plug $L_A(X,Y)$ and $L_B(X,Y)$ into \cref{eq:equiv1_k=l} and obtain the same expressions as in \cref{eq:a_equiv1_1}--\cref{eq:a_equiv1_3} with $u=w$. Analogously to Case~1, the third and the fourth term of \cref{eq:a_equiv1_1} cannot be represented on the right-hand side of \cref{eq:equiv1_k=l} at the same time. Hence, $a_u\overline{a}_w$ has to be zero. Assuming $\overline{a}_w = 0$, we obtain, by similar reasoning as in the previous cases, that $L_A(X,Y)$ and $L_B(X,Y)$ have to be monomials of the same degree as in \cref{eq:LALB_monomials}. Assuming $a_u = 0$, we obtain the same contradiction as in the study of \ref{item:b}, Case~2.\par
		
		In summary, the only possible choice of $L_A(X,Y)$ and $L_B(X,Y)$ that can satisfy \cref{eq:equiv1_k=l} and \cref{eq:equiv2_k=l} is $L_A(X,Y) = a_u X^{2^u}$ and $L_B(x,y)=\overline{b}_u Y^{2^u}$. Considering ~\cref{eq:equiv2_k=l} for these monomials, it follows that $N_2(X) = 0$, $N_4(X) = a_u\overline{b}_uX^{X^{2^u}}$, and $M_B(X,Y) = 0$. If we plug $L_A(X,Y)$ and $L_B(X,Y)$ into~\cref{eq:equiv1_k=l}, we obtain
		\begin{equation}
			\begin{split}
			\label{eq:monomials_in_equiv1}
				&a_u^{2^{2k}(2^k+1)} x^{2^{u+2k}(2^k+1)} + a_u^{2^{2k}} \overline{b}_u^{2^k} x^{2^{u+2k}} y^{2^{u+k}} + \beta \overline{b}_u^{2^k+1} y^{2^u(2^k+1)}
				\\&\qquad = N_1(x^{2^{2k}(2^k+1)} + x^{2^{2k}} y^{2^k} + \beta' y^{(2^k+1)}) + N_3(xy) + M_A(x,y).
			\end{split}
		\end{equation}
		Obviously, $N_3(X) = 0$ and $M_A(X,Y) = 0$ and $N_1(X)$ has to be a monomial of degree $u$, the same degree as $L_A(X,Y)$ and $L_B(X,Y)$. Write $N_1(X) = c_u X^{2^u}$. Then \cref{eq:monomials_in_equiv1} becomes
		\[
		\begin{split}
			&a_u^{2^{2k}(2^k+1)} x^{2^{u+2k}(2^k+1)} + a_u^{2^{2k}} \overline{b}_u^{2^k} x^{2^{u+2k}} y^{2^{u+k}} + \beta \overline{b}_u^{2^k+1} y^{2^u(2^k+1)}
			\\&\qquad = c_ux^{2^{u+2k}(2^k+1)} + c_ux^{2^{u+2k}} y^{u+2^k} + c_u\beta'^{2^u} y^{2^u(2^k+1)}
		\end{split}
		\]
		and the coefficients have to meet the following conditions:
		\begin{align}
			\label{eq:Coefficients_final}
			a_u^{2^{2k}(2^k+1)} &= c_u, & a_u^{2^{2k}} \overline{b}_u^{2^k} &= c_u, & \beta \overline{b}_u^{2^k+1} &= c_u \beta'^{2^u}.
		\end{align}
		
		The first two equations of \cref{eq:Coefficients_final} imply $\overline{b}_u = a_u^{2^{2k}}$ and $c_u = \overline{b}_u^{2^k+1}$. Combining the later result with the third equation of \cref{eq:Coefficients_final}, it follows that $\beta = \beta'^{2^u}$.
	\end{proof}

		From the proof of \autoref{th:Taniguchi_equivalence}, we can deduce the order of the automorphism group of the Taniguchi APN functions. Note that \autoref{th:Taniguchi-automorphismgroup} only holds for $m \ge 4$. For $m=2$, the unique Taniguchi APN function $f_{1,1,1}$ on $\F_{2^4}$ is CCZ-equivalent to the Gold APN function $x \mapsto x^3$. Its automorphism group has order $5760$. If $m = 3$, the unique Taniguchi APN function $f_{1,1,\beta}$ on $\F_{2^6}$ is CCZ-equivalent to the APN function $x \mapsto x^3 + ux^{24} + x^{10}$, where $u$ is primitive in $\F_{2^6}$, that was first given by \textcite{browning2009}. In this case, $|\Aut(f_{1,1,\beta})| = 896$. 
	
	\begin{theorem}
		\label{th:Taniguchi-automorphismgroup}
		Let $m \ge 4$, and let $f_{k,\alpha,\beta}$ be a Taniguchi APN function from \autoref{th:ZhouPottAPN} on $\Fptwom$. Define $\beta' = \frac{\beta}{\alpha^{2^{-k}+1}}$. Then
		\[
			|\Aut_{EL}(f_{k,\alpha,\beta})| = 
			\begin{cases}
				3m(2^m-1)	&\text{if } \alpha = 0 \text{ and } m = 4, \rule[-1em]{0em}{1em}\\
				\frac{3}{2}m(2^m-1)	&\text{if } \alpha = 0 \text{ and } m \ge 5,\rule[-1em]{0em}{1em}\\
				\dfrac{m(2^m-1)}{\min \{u : \beta'^{2^u} = \beta'\}}	&\text{if } \alpha \ne 0
			\end{cases} 
		\]
		and
		\[
			|\Aut(f_{k,\alpha,\beta})| = 
			\begin{cases}
				3m2^{2m}(2^m-1)	&\text{if } \alpha = 0 \text{ and } m =4,\rule[-1em]{0em}{1em}\\
				3m2^{2m-1}(2^m-1)	&\text{if } \alpha = 0 \text{ and } m \ge 5,\rule[-1em]{0em}{1em}\\
				\dfrac{m 2^{2m}(2^m-1)}{\min \{u : \beta'^{2^u} = \beta'\}}	&\text{if } \alpha \ne 0.
			\end{cases}
			\vspace{.5em}
		\]
	\end{theorem}
	\begin{proof}
		We determine $|\Aut_{EL}(f_{k,\alpha,\beta})|$, then $|\Aut(f_{k,\alpha,\beta})|$ follows from \autoref{prop:AutEA_AutEL} and \autoref{lem:CCZ_EA}. If $\alpha = 0$, according to \autoref{prop:alpha=0_ZhouPott}, a Taniguchi APN function $f_{k,0,\beta}$ is linearly equivalent to the Pott-Zhou APN function $g_{k,2k,\beta}$ whose automorphism group was determined by the present authors \cite[Theorem~5.2]{kasperszhou2020}.\par
		
		If $\alpha \ne 0$, we know from \autoref{prop:trivial_equivalences}~(a) that $f_{k,\alpha,\beta}$ is linearly equivalent to $f_{k,1,\beta'}$. We study the case $\alpha = 1$. For $m=4$ the results can be confirmed computationally with \texttt{Magma}~\cite{magma}. Assume $m \ge 5$. Then the proof of \autoref{th:Taniguchi_equivalence} holds. We count the number of equivalence mappings that map $f_{k,1,\beta'}$ on itself. Therefore, we consider the conditions given in~\cref{eq:Coefficients_final} which the coefficients of the linearized monomials $L_A(X,Y)$, $L_B(X,Y)$ and $N_1(X)$ have to meet. We have shown that \cref{eq:Coefficients_final} implies
		\begin{align*}
		\overline{b}_u &= a_u^{2^{2k}}, &c_u &= \overline{b}_u^{2^k+1},& \text{and}&& \beta'^{2^u-1} = 1.
		\end{align*}
		The number of $u$ such that $\beta'^{2^u-1} = 1$ is given by 
		\[
			\frac{m}{\min \{u : \beta'^{2^u} = \beta'\}}.
		\]
		Next, we have $2^m-1$ choices for $a_u$. By choosing $a_u$, the coefficients $\overline{b}_u$ and $\overline{c_u}$ are uniquely determined.
	\end{proof}
	
	From \autoref{th:Taniguchi-automorphismgroup}, we easily deduce the following result about the inequivalence of Taniguchi and Pott-Zhou APN functions. Recall that Pott-Zhou APN functions only exist on $\Fptwom$ where $m$ is even and that we have already solved the case $\alpha = 0$ in \autoref{prop:alpha=0_ZhouPott}.
	
	\begin{corollary}
		\label{cor:inequivalence_Taniguchi_ZhouPott}
		Let $m \ge 4$ be even. Let $f_{k,\alpha,\beta}$, where $\alpha \ne 0$, be a Taniguchi APN function from \autoref{th:TaniguchiAPN} on $\Fptwom$, and let $g_{\ell,s,\gamma}$ be a Pott-Zhou APN function from \autoref{th:ZhouPottAPN} on $\Fptwom$. Then $f_{k,\alpha,\beta}$ and $g_{\ell,s,\gamma}$ are CCZ-inequivalent.
	\end{corollary}
	\begin{proof}
		The order of the automorphism group of a vectorial Boolean function is invariant under CCZ-equivalence. For a Taniguchi APN function $f_{k, \alpha, \beta}$ on $\Fptwom$, we determined the order of the automorphism group $\Aut(f_{k, \alpha, \beta})$ in \autoref{th:Taniguchi-automorphismgroup}. For a Pott-Zhou APN function $g_{\ell,s,\gamma}$ on $\Fptwom$, the present authors~\cite[Theorem~5.2]{kasperszhou2020} showed that
		\[
			|\Aut(g_{\ell,s,\gamma})| = 
			\begin{cases}
			3m2^{2m}(2^m-1)	&\text{if } s \in \{0, \frac{m}{2}\},\\
			3m2^{2m-1}(2^m-1)	&\text{otherwise}.\\
			\end{cases} 
		\]
		As clearly $\frac{m}{\min \{u : \beta'^{2^u} = \beta'\}} \le m$, it follows that $\frac{m}{\min \{u : \beta'^{2^u} = \beta'\}} < \frac{3}{2}m < 3m$. Hence, the automorphism groups of $f_{k,\alpha,\beta}$ and $g_{\ell,s,\gamma}$ are of different order which implies that the functions are CCZ-inequivalent.
	\end{proof}

	From \autoref{cor:inequivalence_Taniguchi_ZhouPott}, we derive the final piece to determine the complete equivalence of Taniguchi APN functions.
	\begin{corollary}
		\label{cor:inequivalence_Taniguchi_alpha=0_1}
		Let $m \ge 4$ be even. Two Taniguchi APN functions $f_{k,0,\beta}$, and $f_{\ell,\alpha',\beta'}$ , where $\alpha' \ne 0$, from \autoref{th:TaniguchiAPN} on $\Fptwom$ are CCZ-inequivalent.
	\end{corollary}
	\begin{proof}
		According to \autoref{prop:alpha=0_ZhouPott}, $f_{k,0,\beta}$ is CCZ-equivalent to a Zhou-Pott APN function $g_{k,2k,\gamma}$ from \autoref{th:ZhouPottAPN}. The result now follows from \autoref{cor:inequivalence_Taniguchi_ZhouPott}.
	\end{proof}
	
\section{On the total number of CCZ-inequivalent Taniguchi APN functions on $\Fptwom$}
\label{sec:Taniguchi_number}

	The results from \autoref{sec:Taniguchi_equivalence} allow us to determine the number of CCZ-inequivalent Taniguchi APN functions on $\Fptwom$ for any $m$. This will be done in \autoref{th:numberTaniguchi} by counting the number of parameters $k$, $\alpha$ and $\beta$ that lead to inequivalent functions. Recall from \autoref{prop:trivial_equivalences} that every Taniguchi APN function $f_{k,\alpha,\beta}$ where $\alpha \ne 0$ is CCZ-equivalent to a function $f_{k,1,\beta'}$ for some $\beta' \in \Fpm^*$. Hence, we only need to consider functions with $\alpha = 0$ or $\alpha = 1$. As we know from \autoref{prop:alpha=0_ZhouPott} that $f_{k,0,\beta}$ is equivalent to a Pott-Zhou APN function, whose equivalence problem was solved by the present authors~\cite{kasperszhou2020}, we focus on $\alpha = 1$ first.\par
	
	Recall from \autoref{th:Taniguchi_equivalence} that two Taniguchi APN functions $f_{k,1,\beta}$ and $f_{k,1,\beta'}$ on $\Fptwom$ are CCZ-equivalent if and only if $\beta' = \beta^{2^i}$ for some $i \in \{0, \dots, m-1\}$. Consequently, to obtain the exact number of $\beta$ providing inequivalent functions for fixed $k$, we need to determine the number of orbits of $\beta$ such that $X^{2^k+1} + X + \beta$ has no root in $\Fpm$ under the action of the Galois group $\Gal(\Fpm / \F_2)$.  We will do this in \autoref{prop:beta_number} with the help of the following series of technical lemmas.
	
	\begin{lemma}\label{lem:3k}
		If $k>1$ is an integer with $\gcd(k,3) = 1$, then $3k$ does not divide $2^{k}+1$.
	\end{lemma}
	\begin{proof}
		Assume, by way of contradiction, that $3k \mid 2^{k}+1$. By the Chinese Remainder Theorem, $2^{k} \equiv -1 \pmod{3}$ which means that $k$ is odd. 
		
		Let $k=p_1^{t_1}\cdots p_s^{t_s}$, where $p_1, \dots, p_s$ are prime numbers such that $3<p_1<p_2<\cdots<p_s$ and $t_i\geq 1$ for $i = 1, \dots, s$. For convenience, we set $p=p_1$ and $t=t_1$ in the remainder of this proof. 
		
		By the Chinese Remainder Theorem, it also follows that $2^k\equiv -1 \pmod{p^t}$. Denote by $\varphi(x)$ the Euler's totient function of $x$. Since $2^{2^k} \equiv 1 \pmod{p^t}$ and the unit group of the integer ring $\Z_{p^t}$ has order $\varphi(p^t)$, it follows that $\mathrm{ord}_{p^{t}}(2) \mid \gcd(2k, \varphi(p^{t}))$. Note that $\varphi(p^t) = (p-1)p^{t-1}$. As $p-1 < p_i$ for all $i \in \{1,\dots, s\}$, the number $p-1$ is not divisible by any of the $p_i$. Recalling that $k = p^{t} p_2^{t_2}\cdots p_s^{t_s}$, it follows that $\gcd(2k,\varphi(p^t))=2p^{t-1}$. Consequently, $2^{2p^{t-1}}-1\equiv 0 \pmod{p^t}$. Thus $2^{2p^{t-1}}-1\equiv 4^{p^{t-1}}-1\equiv 0 \pmod{p}$. As $4^{p}=4\pmod{p}$, we obtain $4-1\equiv 0\pmod{p}$ which means $p=3$. This is a contradiction to the assumption $3 < p$.
	\end{proof}

	\begin{lemma}\label{lem:zeros_3}
		Suppose that $k$ and $m$ are positive integers satisfying $\gcd(k,m)=1$. Write $m=rp$ for an integer $r$ and a prime $p$. For $\beta\in \F_{2^r}$, suppose that the polynomial $P(X)=X^{2^k+1}+X+\beta$ has no root in $\F_{2^r}$.
		\begin{enumerate}[label=(\alph*)]
			\item\label{item:zeros_a} If $p\neq 3$, then $P(X)$ has no root in $\Fpm$.
			\item\label{item:zeros_b} If $p= 3$, then $P(X)$ has exactly three roots in $\Fpm$.
		\end{enumerate}
	\end{lemma}
	\begin{proof} Set $\sigma(x)=x^{2^{r}}$ for $x$ in any extension of $\F_{2^r}$.
		
		We show \ref{item:zeros_a} first. Suppose that $P(X)$ has at least one root $x_0 \in \Fpm$. Then $x_0, \sigma(x_0), \dots, \sigma^{p-1}(x_0)$ have to be $p$ distinct roots of $P(X)$ in $\Fpm$ because $\sigma(P(x_0))= \sigma(x_0)^{2^k+1}+\sigma(x_0)+\beta=0$ and $p$ is prime. \textcite[Theorem~1]{hellesethkholosha2008} showed that if $P(X)$ has more than one root, then $P(X)$ has exactly three roots in $\Fpm$ which contradicts the assumption that $p\neq3$.\par
		
		We next prove \ref{item:zeros_b}. Now $m=3r$. If $P(X)$ has at least one root in $\Fpm$, by the proof of~\ref{item:zeros_a}, it has exactly three roots in $\Fpm$ and we are done. Assume, by way of contradiction, that $P(X)$ has no root in $\Fpm$. First, if $k=1$, then $P(X)$ has degree $3$ and is irreducible over $\F_{2^r}$. Therefore, $P(X)$ splits over $\Fpm$ which contradicts our assumption.\par
		
		From now on, assume $k>1$. Write $P(X) = P_1(X)P_2(X) \cdots P_s(X)$ for irreducible polynomials $P_1(X), \dots, P_s(X) \in \Fpm$. Since $\deg(P(X)) = 2^k+1$ is odd, there exists a polynomial $P_j(X)$, where $j \in \{1,\dots, s\}$, of odd degree. Denote by $J_{\textnormal{odd}}$ the set of all $j \in \{1,\dots,s\}$ such that $\deg(P_j(X))$ is odd, and let $j^* \in J_{\textnormal{odd}}$ such that $\deg(P_{j^*}(X)) \le \deg(P_j(X))$ for all $j \in J_{\textnormal{odd}}$. Set $\ell = \deg(P_{j^*}(X))$ and note that $\ell > 1$ and $\ell$ is odd. Then $P_{j^*}(X)$ splits over  $\F_{2^{m\ell}}$, which is an extension of $\Fpm$ with $[\F_{2^{m\ell}}:\Fpm]=\ell$. Consequently, $P(X)$ has a root in $\F_{2^{m\ell}}$, and there is no root of $P(X)$ in any proper subfield of $\F_{2^{m\ell}}$ containing $\Fpm$.\par
		Define $h = \gcd(m\ell, k)$. As $m$ and $k$ are coprime, this implies $h = \gcd(\ell, k)$ and, in particular, $h \mid \ell$. Then $\F_{2^h}=\F_{2^{m\ell}}\cap \F_{2^k}$. As $\ell$ is odd, according to \textcite[Theorem~5.6]{bluher2004}, $P(X)$ has exactly $2^{h}+1$ roots in $\F_{2^{m\ell}}$. If $h = 1$, then the roots of $P(X)$  in $\F_{2^{m\ell}}$ are also elements of $\Fpm$ as $m=3r$. This contradicts our assumption. Hence, assume $h > 1$. We may regard $\sigma$ as an element in $\Gal(\F_{2^{m\ell}}/\F_{2^r})$. If $3\nmid \ell$, then it is clear that $x_0$, $\sigma(x_0)$, $\dots, \sigma^{3\ell}(x_0)$ are pairwise distinct for any root~$x_0$ of $P(X)$ in $\F_{2^{m\ell}}$. If $3\mid \ell$, then $x_0$, $\sigma(x_0)$, $\dots, \sigma^{3\ell}(x_0)$ are still pairwise distinct for any root $x_0$ of $P(X)$ in $\F_{2^{m\ell}}$. The reason is as follows. Suppose that $\sigma^{j}(x_0)=x_0$ for some $j<3\ell$ with $j\mid 3\ell$. This means $[\F_{2^r}(x_0): \F_{2^r}]=j$. Thus,
		\[
			[\Fpm(x_0):\Fpm] =
			\begin{cases}
				j &\text{if } 3\nmid j,\\
				j/3 &\text{if } 3\mid j.
			\end{cases}
		\]
		For the first case, $3\nmid j$, as $\F_{2^{m\ell}}=\Fpm(x_0)$ by definition, we get $j=\ell$ which is a contradiction to the assumption that $3\mid \ell$. For the second case, $3\mid j$, we get $\ell = [\F_{2^{m\ell}}:\Fpm] = [\Fpm(x_0):\Fpm] = j/3$ which contradicts the assumption $j<3\ell$.
		
		Therefore, $3\ell$ divides $2^h+1$, in particular, as $h \mid \ell$, we obtain  $3h \mid 2^h+1$. By \autoref{lem:3k}, this is only possible if $\gcd(h,3) > 1$ which implies $\gcd(m,k) > 1$. This is a contradiction.
	\end{proof}
	
	For any two relatively prime positive integers $k$ and $m$, define
	\begin{equation}
	\label{eq:Phi}
		\Phi(m) = \{\beta\in \F_{2^m}: X^{2^k+1}+X+\beta \text{ has no roots in }\F_{2^m}\}
	\end{equation}
	and
	\[
		M(m) =|\Phi(m)|
	\] 
	and
	\begin{equation}
	\label{eq:N(m)_definition}
		N(m) = \left|\{\beta\in \Phi(m): \beta\notin\F_{2^{m'}} \text{ with } m'< m \text{ and } m'\mid m\}\right|.
	\end{equation}
	According to \autoref{lem:number_of_beta}, 
	\begin{equation}
	\label{eq:M(m)}
		M(m)=\frac{2^m+(-1)^{m+1}}{3}. 
	\end{equation}
	In the following \autoref{lem:N(m)}, we determine the exact value of $N(m)$.
	
	\begin{lemma}
	\label{lem:N(m)}
		Suppose that $m=3^{n_0}\prod_{i=1}^t p_i^{n_i}$ where $n_0$ is a non-negative integer,  $p_1, \dots, p_t$ are distinct prime numbers, and $n_1, \dots, n_t$ are positive integers. If $t = 0$, that means $m = 3^{n_0}$ and, in particular, includes the case $m=1$, then 
		\[
			N(m)=\frac{2^{m}+1}{3}.
		\]
		If $t \ge 1$, then
		\begin{equation}
			\begin{split}
				N(m)& = \frac{1}{3} \Bigg( 2^m -\sum_{i=1}^t 2^{\frac{m}{p_i}} + \sum_{\substack{i,j=1,\\j \ne i}}^t 2^{\frac{m}{p_i p_j}}- \dots\\
				\label{eq:N(m)} 
				&\quad \dots + (-1)^\ell \sum_{\substack{i_1,\dots,i_\ell = 1\\\text{pairwise distinct}}}^t 2^{\frac{m}{p_{i_1} \cdots p_{i_\ell}}}+ \dots + (-1)^t2^{\frac{m}{p_1 p_2 \cdots p_t}}-\varepsilon\Bigg),
			\end{split}
		\end{equation}
		where
		\[
			\varepsilon=
			\begin{cases}
				2 & \text{if } t=1
				 \text{ and } m \equiv 2 \pmod{4},\\
				0 &\text{otherwise}.
			\end{cases}
		\]
	\end{lemma}
	\begin{proof}
		By definition, to determine $N(m)$, we have to exclude each element in $\Phi(m)\cap \F_{2^{m'}}$ from $\Phi(m)$ for every proper subfield $\F_{2^{m'}}$ of $\F_{2^m}$. We first consider the case $t=0$: If $n_0 = 1$, which means $m=1$, then $X^{2^k+1} + X + \beta$ has no root in $\F_2$ if and only if $\beta = 1$. Hence, $N(1) = 1$. If $n_0 \ge 1$, by \autoref{lem:zeros_3}, 
		\[
			\Phi(m)\cap \F_{2^{m'}}=
			\begin{cases}
				\emptyset & \text{if } 3m' \mid m,\\
				\Phi(m') & \text{if } 3m' \nmid m.		
			\end{cases}
		\]
		Hence, we get $N(3^{n_0})=M(3^{n_0})$ and, by \cref{eq:M(m)}, $M(3^{n_0}) = \frac{2^m+1}{3}$. From now on, assume $t \ge 1$. Then, by the inclusion-exclusion principle, 
		\begin{equation}
		\begin{split}
		\label{eq:N(m)_ex}
			N(m)& =M(m)-\sum_{i=1}^t M\left(\frac{m}{p_i}\right) + \sum_{\substack{i,j=1,\\j \ne i}}^t M\left(\frac{m}{p_ip_j}\right) - \cdots\\
			&\quad \dots +(-1)^\ell\sum_{\substack{i_1,\cdots,i_\ell = 1 \\ \text{pairwise distinct}}}^t M\left(\frac{m}{p_{i_1}\cdots p_{i_\ell}}\right) + \cdots + (-1)^t M\left(\frac{m}{p_1\cdots p_t}\right).
		\end{split}
		\end{equation}
		
		If $m$ is odd, then $m'$ is odd for all $m' \mid m$. If $4 \mid m$, then $m'$ is even for all $m' = \frac{m}{p_{i_1} \cdots p_{i_\ell}}$ that occur in \cref{eq:N(m)_ex}. Consequently, in these two cases, by \eqref{eq:M(m)}, we have $M(m')=\frac{2^{m'}+(-1)^{m+1}}{3}$ for any $m' = \frac{m}{p_{i_1} \cdots p_{i_\ell}}$ occuring in \eqref{eq:N(m)_ex}. Plugging $M(m')$ into \eqref{eq:N(m)_ex}, we obtain
		\begin{equation}
			\begin{split}
			\label{eq:N1}
				N(m)=&\frac{1}{3}\Bigg( 2^m - \sum_{i=1}^t 2^{\frac{m}{p_i}} + \sum_{\substack{i,j=1,\\j \ne i}}^t 2^{\frac{m}{p_i p_j}} - \cdots + (-1)^t 2^{\frac{m}{p_1p_2\cdots p_t}} \Bigg)\\
				&\quad + \frac{(-1)^{m+1}}{3}\left(1-\binom{t}{1}+\binom{t}{2}-\cdots+(-1)^t  \right).
			\end{split}
		\end{equation}
		Note that the last sum of \cref{eq:N1} equals zero which can be seen by using the binomial identity 
		\[
			(x+y)^n = \sum_{k=0}^{n} \binom{n}{k} x^{n-k} y^k
		\]
		with $x=1$ and $y=-1$ (or vice versa).\par
		
		If $m \equiv 2 \pmod{4}$, we set $p_1=2$ and $n_1=1$. By \cref{eq:M(m)},
		\begin{equation}
		\label{eq:M(m')}
			M(m') =
			\begin{cases}
				\frac{2^{m'}+1}{3}	&\text{if } m' = \frac{m}{2p_{i_2} \cdots p_{i_\ell}},\\
				\frac{2^{m'}-1}{3}	&\text{if } m' = \frac{m}{p_{i_1} \cdots p_{i_\ell}} \text{ and } i_1, \dots, i_\ell \ne 1.
			\end{cases}
		\end{equation}
		Plugging \cref{eq:M(m')} into \cref{eq:N(m)_ex}, we obtain
		\begin{equation}
			\begin{split}
			\label{eq:N(m)_m2mod4}
				N(m)=&\frac{1}{3}\Bigg( 2^m - \sum_{i=1}^t 2^{\frac{m}{p_i}} + \sum_{\substack{i,j=1,\\j \ne i}}^t 2^{\frac{m}{p_i p_j}} - \cdots + (-1)^t 2^{\frac{m}{p_1p_2\cdots p_t}} \Bigg)\\
				&\quad + \frac{1}{3} \sum_{i=0}^t (-1)^i \left(\binom{t-1}{i-1} - \binom{t-1}{i}\right).
			\end{split}
		\end{equation}
		
		We show where the last sum of \cref{eq:N(m)_m2mod4} is coming from and which values it can take. If $t=1$, then $m=3^{n_0} \cdot 2$. Note that $m$ is even and $\frac{m}{2}$ is odd. Hence, in this case, $N(m) = M(m) - M(\frac{m}{2}) = 2^m - 2^\frac{m}{2} -2$, and the last sum of \cref{eq:N(m)_m2mod4} equals $-2$. Now assume~$t > 1$. Consider the sum
		\begin{equation}
		\label{eq:sum}
			\sum_{\substack{i_1,\cdots,i_\ell = 1 \\ \text{pairwise distinct}}}^t M\left(\frac{m}{p_{i_1}\cdots p_{i_\ell}}\right)
		\end{equation}
		from \cref{eq:N(m)} for some $\ell \in \{1, \dots, t\}$. This sum consists of $\binom{t}{\ell}$ terms. Assume $p_{i_1} < p_{i_2} < \dots < p_{i_\ell}$. If $i_1 = 1$, which means $p_{i_1} = 2$, then $\frac{m}{2 p_{i_2}\cdots p_{i_\ell}}$ is odd. In this case, we have $\binom{t-1}{\ell-1}$ possibilities to choose $p_{i_2}, \dots, p_{i_\ell}$. On the contrary, if $i_1 \ne 1$, then $\frac{m}{p_{i_1}\cdots p_{i_\ell}}$ is even, and we have $\binom{t-1}{\ell}$ possibilities to choose $p_{i_1}, \dots, p_{i_\ell}$. Combining these results with \cref{eq:M(m')}, we have $\binom{t-1}{\ell-1}$ terms of the form $(2^{m'} + 1)$ and $\binom{t-1}{\ell}$ terms of the form $(2^{m'} - 1)$ in the sum from \cref{eq:sum}. Note that, by similar reasoning as in the case $m$ odd or $4 \mid m$, this sum is zero if $t > 1$.
	\end{proof}
	
	Consider $\Phi(m)$ as in \cref{eq:Phi}. We have shown in \autoref{lem:polynomial_transformation} that if $X^{2^k+1} + X + \beta$ has no root in $\Fpm$, then neither has $X^{2^k+1} + X + \beta^{2^i}$ for all $i \in \{0,\dots,m-1\}$. Consequently, $\Phi(m)$ decomposes into orbits of $\beta \in \Fpm^*$ under the action of the Galois group $\Gal(\Fpm / \F_2)$. In \autoref{prop:beta_number}, we count this number of orbits.
	
	\begin{proposition}
	\label{prop:beta_number}
		Let $\Phi(m)$ as in \cref{eq:Phi}, and define
		\[
			B(m) = \left\{ \{ \beta^{2^i} : i \in \{0,\dots,m-1\} \} : \beta \in \Phi(m)\right\}
		\]
		as the set of orbits of $\beta \in \Fpm^*$ such that $X^{2^k+1} + X + \beta$ has no root in $\Fpm$ under the action of the Galois group $\Gal(\Fpm / \F_2)$. Moreover, define $b(m) = |B(m)|$. Then
		\[
			b(m) = \sum_{m' \mid m,\ 3 \nmid \frac{m}{m'}} \frac{N(m')}{m'},
		\]
		where $N(m')$ is defined as in \cref{eq:N(m)_definition} and can be calculated as in \autoref{lem:N(m)}.
	\end{proposition}
	\begin{proof}
		For any subfield $\F_{2^{m'}}$ of $\Fpm$, we count the number of orbits of  $\beta \in \Phi(m) \cap \F_{2^{m'} }^*$ under the action of $\Gal(\F_{2^{m'}}/\F_2)$ that have full length~$m'$. This number is given by $\frac{N(m')}{m'}$. It follows from \autoref{lem:zeros_3} that we only need to consider the orbits in $\F_{2^{m'}}$ with $3 \nmid [\Fpm : \F_{2^{m'}}]$. Adding all these numbers gives $b(m)$. 
	\end{proof}

	With the help of \autoref{prop:beta_number}, we can eventually determine the number of CCZ-inequivalent Taniguchi APN functions on $\Fptwom$ in \autoref{th:numberTaniguchi}. We give a nice lower bound on this number in \autoref{cor:bound}.
	
	\begin{theorem}
	\label{th:numberTaniguchi}
		Let $m \ge 3$, and denote by $n(m)$ the number of CCZ-inequivalent Taniguchi APN functions $f_{k,\alpha,\beta}$ from \autoref{th:TaniguchiAPN} on $\Fptwom$. Then
		\[
			n(m) = 
			\begin{cases}
			\dfrac{\varphi(m)b(m)}{2}	&\text{if $m$ is odd},\rule[-1.2em]{0em}{1em}\\
			\dfrac{\varphi(m)(b(m)+1)}{2}	&\text{if $m$ is even},
			\end{cases}
		\]
		where $\varphi$ denotes Euler's totient function and $b(m)$ is as in \autoref{prop:beta_number}.
	\end{theorem}
	\begin{proof}
		Let $m \ge 3$. Thanks to \autoref{prop:trivial_equivalences}, we only need to consider $\alpha \in \{0,1\}$ and $0 < k < \frac{m}{2}$. We count the number of CCZ-inequivalent Taniguchi APN functions $f_{k,1,\beta}$ first: According to \autoref{th:Taniguchi_equivalence}, for $0 < k,\ell <\frac{m}{2}$ two functions $f_{k,1,\beta}$ and $f_{\ell,1,\beta'}$ are CCZ-equivalent if and only if $k = \ell$ and $\beta = \beta'^{2^i}$ for some $i \in \{0,\dots,m-1\}$. We count the number of pairs $(k,\beta)$ that lead to inequivalent APN functions: As $0 < k < \frac{m}{2}$ and $\gcd(k,m)=1$, we have $\frac{\varphi(m)}{2}$ choices for $k$. The number of admissible $\beta \in \Fpm^*$ equals $b(m)$ from \autoref{prop:beta_number}. If $m$ is odd, then these are all inequivalent Taniguchi APN functions.\par
		
		If $m$ is even, according to \autoref{lem:alpha=0}, there also exist Taniguchi APN functions with $\alpha = 0$. In this case, it follows from \autoref{cor:inequivalence_Taniguchi_alpha=0_1} in combination with \autoref{cor:alpha=0} that for every valid choice of $k$, there is additionally exactly one equivalence class of Taniguchi APN functions $f_{k,0,\beta}$, that is inequivalent to all functions with $\alpha \ne 0$. As before, we have $\frac{\varphi(m)}{2}$ choices for $k$.
	\end{proof}

	Note that \autoref{cor:bound} shows that the number of APN functions on $\Fptwom$ increases exponentially in $m$.
	\begin{corollary}
	\label{cor:bound}
		Let $m \ge 3$, and define $n(m)$ as the number of CCZ-inequivalent Taniguchi APN functions from \autoref{th:TaniguchiAPN} on $\Fptwom$. Then
		\[
			n(m) \ge \frac{\varphi(m)}{2}\left\lceil\frac{2^m+1}{3m}\right\rceil,
		\]
		where $\varphi$ denotes Euler's totient function.
	\end{corollary}
	\begin{proof}
		Define $B(m)$ and $b(m)$ as in \autoref{prop:beta_number}. The value of $b(m)$ is minimal if all the orbits in $B(m)$ have full length $m$. By \autoref{lem:number_of_beta}, this implies 
		\[
			b(m) \ge 
			\begin{cases}
				\left\lceil\frac{2^m-1}{3m}\right\rceil	&\text{if $m$ is even} \rule[-1.2em]{0em}{1em},\\
				\left\lceil\frac{2^m+1}{3m}\right\rceil	&\text{if $m$ is odd},
			\end{cases}
		\]
		and it is easy to see that $\left\lceil\frac{2^m-1}{3m}\right\rceil = \left\lceil\frac{2^m+1}{3m}\right\rceil$ for all $m \ge 3$.
	\end{proof}

	In \autoref{tab:Taniguchi-APN_smallm}, we list the exact number of CCZ-inequivalent Taniguchi APN functions obtained from \autoref{th:numberTaniguchi} for certain values of $m$. Recall that for $m=2$, there is only one unique Taniguchi APN function. We moreover compare these numbers to the lower bound that we have established in \autoref{cor:bound}. It can be seen that the bound is very close to the actual number of Taniguchi APN functions.\par
	
	\begin{table}[h]
		\caption{Number of CCZ-inequivalent Taniguchi APN functions on $\Fptwom$ for certain values of $m$.}
		\label{tab:Taniguchi-APN_smallm}		
		\begin{tabular}{*{16}{r}}
			\hline
			m &2 &3 &4 &5 &6 &7 &8 &9 &10 &11 &12 &13 &14 &15 &16\rule[.5em]{0em}{.5em}\\\hline
			\# 	&1& 1& 3& 6& 5& 21& 26& 57& 74& 315& 234& 1\,266& 1\,185& 2\,916	&5\,492\rule[.5em]{0em}{.5em}\\
			bound&1& 1& 2& 6& 4& 21& 22& 57& 70& 315& 228& 1\,266& 1\,173& 2\,916 &5\,464\rule[.5em]{0em}{.5em}\\\hline
		\end{tabular}\par
		\vspace{1em}
	
		\begin{tabular}{*{8}{r}}
			\hline
			m	&17 &18 &19 &20 &25	&50	&100\rule[.5em]{0em}{.5em}\\\hline
			\# 		& 20\,568& 14\,595& 82\,791& 69\,988& 4\,473\,950& $\approx 7.5 \cdot 10^{13}$& $\approx 8.5 \cdot 10^{28}$\rule[.5em]{0em}{.5em}\\
			bound	& 20\,568& 14\,565& 82\,791&	69\,908& 4\,473\,930& $\approx 7.5 \cdot 10^{13}$& $\approx 8.5 \cdot 10^{28}$\rule[.5em]{0em}{.5em}\\\hline
		\end{tabular}
	\end{table}

\section{Conclusion and open questions}
\label{sec:conclusion}
	In the present paper, we establish a new lower bound on the total number of CCZ-inequivalent APN functions on the finite field $\Fptwom$. We show that the number of APN functions on $\Fptwom$ grows exponentially in $m$. For even $m$, our result presents a great improvement of the lower bound previously given by the present authors~\cite{kasperszhou2020}. For odd~$m$, this is the first such lower bound.\par
	
	Our result now shifts the focus on the following open problems concerning APN functions:
	\begin{itemize}
		\item Establish a lower bound on the total number of CCZ-inequivalent APN functions on the finite field $\F_{2^n}$ with $n$ odd.
		\item As it is confirmed now that there are very many quadratic APN functions on~$\Fptwom$, the efforts of finding new constructions of APN functions should focus on the search for non-quadratic ones.
		\item It was shown by \textcite{anbar2019} that Taniguchi APN functions have the classical Walsh spectrum. It would be interesting to find more APN functions with non-classical Walsh spectra.
	\end{itemize}
	
\section*{Acknowledgments}
We thank the anonymous reviewers for their useful comments and suggestions, and we thank Satoshi Yoshiara and Ulrich Dempwolff for their helpful comments on \autoref{lem:CCZ_EA} and the connection of the automorphism groups of quadratic APN functions under EA- and under CCZ-equivalence.\par

This work is partially supported by National Natural Science Foundation of China (Nos.\ 11771451, 617722213) and Training Program for Excellent Young  Innovators of Changsha (No. kq1905052).

\printbibliography
	
\end{document}